\documentclass[a4paper, 11pt]{amsart}
\usepackage[scale=0.72,centering]{geometry}
\usepackage{amsmath,amssymb,amsthm,graphicx,bbm,pdfpages}
\usepackage{color,CJK,mathrsfs}
\usepackage{hyperref}
\usepackage[font={footnotesize}]{caption}
\numberwithin{equation}{section}
\newtheorem{theorem}{\textbf{Theorem}}[section]

\newtheorem{lemma}[theorem]{\textbf{Lemma}}
\newtheorem{corollary}[theorem]{\textbf{Corollary}}
\newtheorem{proposition}[theorem]{\textbf{Proposition}}
\newtheorem{remark}{\textbf{Remark}}[section]
\begin{document}
\author{RAN WEI}
\thanks{Department of Mathematics, National University of Singapore, 10 Lower Kent Ridge Road, 119076 Singapore. Email address: weiran@u.nus.edu}
\title{FREE ENERGY OF THE CAUCHY DIRECTED POLYMER MODEL AT HIGH TEMPERATURE}
\date{}
\begin{abstract}
We study the Cauchy directed polymer model on $\mathbb{Z}^{1+1}$, where the underlying random walk is in the domain of attraction to the $1$-stable law. We show that, if the random walk satisfies certain regularity assumptions and its symmetrized version is recurrent, then the free energy is strictly negative at any inverse temperature $\beta>0$. Moreover, under additional regularity assumptions on the random walk, we can identify the sharp asymptotics of the free energy in the high temperature limit, namely, 
\begin{equation*}
\lim\limits_{\beta\to0}\beta^{2}\log(-p(\beta))=-c.
\end{equation*}
\end{abstract}
\maketitle
\section{Introduction}
In this paper, we study a specific long-range directed polymer model, that is, the Cauchy directed polymer model on $\mathbb{Z}^{1+1}$. The long-range directed polymer model is an extension of the classic nearest-neighbor directed polymer model. For details about the nearest-neighbor model, we refer to \cite{Com17,CSY04, Hol09}; for details about the long-range model, we refer to \cite{Com07,Wei16}.
\subsection{The model}\label{s101}
We now introduce the Cauchy directed polymer model on $\mathbb{Z}^{1+1}$. The model consists of a random field and a heavy-tailed random walk on $\mathbb{Z}$, whose increment distribution is in the domain of attraction of the $1$-stable law. The random field models the random environment and the random walk models the polymer chain. The polymer chain interacts with the random environment.  We want to investigate whether this interaction significantly influences the behavior of the polymer chain compared to the case with no random environment.

To be precise, we denote the random walk, its probability and expectation by $S=(S_{n})_{n\geq0}, \mathbf{P}$, and $\mathbf{E}$ respectively. The random walk $S$ starts at $0$ and has i.i.d.\ increments satisfying
\begin{equation}\label{e101}
\begin{cases}
&\mathbf{P}(S_{1}-S_{0}>k)/\mathbf{P}(|S_{1}-S_{0}|>k)\sim p,\\
&\mathbf{P}(|S_{1}-S_{0}|>k)\sim k^{-1}L(k),
\end{cases}
\quad\quad\mbox{for some}~p\in[0,1]~\mbox{as}~k\to\infty,
\end{equation}
where $L(\cdot)$ is a function slowly varying at infinity, i.e., $L(\cdot):(0,\infty)\to(0,\infty)$ and for any $a>0, \lim_{t\to\infty}L(at)/L(t)=1$. The condition \eqref{e101} is necessary and sufficient for $S_{1}$ to belong to the domain of attraction of the $1$-stable law, i.e.\, there exist a positive sequence $(a_{n})_{n\geq1}$ and a real sequence $(b_{n})_{n\geq1}$, such that
\begin{equation}\label{e102}
\frac{S_{n}-b_{n}}{a_{n}}\overset{d}{\to}G,\quad\mbox{as}~n\to\infty,
\end{equation}
where $\overset{d}{\to}$ stands for weak convergence and $G$ is some $1$-stable random variable. When $G$ is symmetric, it is known as the Cauchy distribution. In this paper, with a slight abuse of terminology, we say that any $1$-stable law is \textit{Cauchy} for convenience.

Convergence \eqref{e102} is a well-known result. It can be shown that
\begin{align}
\label{e103}&n\mathbf{P}(|S_{1}|>a_{n})\sim 1,\quad\mbox{as}~n\to\infty.\\
\label{e104}&b_{n}=\begin{cases}
n\mathbf{E}[S_{1}],&\mbox{if}~\mathbf{E}[|S_{1}|]<\infty,\\
n\mathbf{E}[S_{1}\mathbbm{1}_{\{|S_{1}|\leq a_{n}\}}],~&\mbox{if}~\mathbf{E}[|S_{1}|]=\infty.
\end{cases}
\end{align}
Furthermore, we have:
\begin{equation}\label{e105}
a_{n}=n\varphi(n)
\end{equation}
with $\varphi(n)=n^{-1}\sup\{x: x^{-1}nL(x)\leq1\}$, where $\varphi(\cdot)$ can be proved to be slowly varying at infinity.

The random field, its probability and expectation are denoted by $\omega:=(\omega_{n,x})_{n\in\mathbb{N}, x\in\mathbb{Z}}, \mathbb{P}$ and $\mathbb{E}$ respectively. Here $\omega$ is a family of i.i.d. random variables independent of the random walk $S$. We assume that $\omega$'s moment generating function is finite in a neighborhood of $0$, meaning that there exists a constant $c>0$, such that
\begin{equation}\label{e106}
\lambda(\beta):=\log\mathbb{E}[\exp(\beta\omega_{n,x})]<\infty,\quad\forall\beta\in(-c,c).
\end{equation}
Beside \eqref{e106}, we also assume that
\begin{equation}\label{e107}
\mathbb{E}[\omega_{n,x}] = 0\quad\mbox{and}\quad\mathbb{E}[(\omega_{n,x})^{2}]=1.
\end{equation}
Given the random environment $\omega$ and polymer length $N$, the law of the polymer is defined via a Gibbs transformation of the law of the underlying random walk, namely, 
\begin{equation*}
\frac{\mbox{d}\mathbf{P}_{N,\beta}^{\omega}}{\mbox{d}\mathbf{P}}(S):=\frac{1}{Z_{N,\beta}^{\omega}}
\exp\left(\sum\limits_{n=1}^{N}\beta\omega_{n,S_{n}}\right),
\end{equation*}
where $\beta>0$ is the inverse temperature and
\begin{equation*}
Z_{N,\beta}^{\omega}=\mathbf{P}\left[\exp\left(\sum\limits_{n=1}^{N}\beta\omega_{n,S_{n}}\right)\right]
\end{equation*}
is the partition function.

It turns out that $Z_{N,\beta}^{\omega}$ plays a key role in the study of the directed polymer model. In \cite{Bol89}, Bolthausen first showed that the normalized partition function 
\begin{equation*}
\hat{Z}_{N,\beta}^{\omega}:=\exp(-N\lambda(\beta))Z_{N,\beta}^{\omega}
\end{equation*}
converges to a limit $\hat{Z}_{\infty,\beta}^{\omega}$ almost surely with either $\mathbb{P}(\hat{Z}_{\infty,\beta}^{\omega}=0)=0$ or $\mathbb{P}(\hat{Z}_{\infty,\beta}^{\omega}=0)=1$ (depending on $\beta$). The range of $\beta$ satisfying the former is called \textit{the weak disorder regime} and the range of $\beta$ satisfying the latter is called \textit{the strong disorder regime}. It has been shown (\textit{cf}. \cite{CY06,Wei16}) that in the weak disorder regime, the polymer chain still fluctuates on scale $a_{n}$, similar to the underlying random walk. This phenomenon is called \textit{delocalization}. It is believed that in the strong disorder regime, there should be a narrow corridor in space-time with distance to the origin much larger than $a_{n}$ at time $n$, to which the polymer chain is attracted with high probability. This phenomenon is called \textit{localization}.

There actually exists a stronger condition than strong disorder, which we now introduce. As in the physics literature, we define \textit{the free energy of the system} by
\begin{equation}\label{e108}
p(\beta):=\lim\limits_{N\to\infty}\frac{1}{N}\log\hat{Z}_{N,\beta}^{\omega}.
\end{equation}
Celebrated results like \cite[Proposition 2.5]{CSY03} and \cite[Proposition 3.1]{Com07} show that the limit in \eqref{e108} exists almost surely and
\begin{equation}\label{e109}
p(\beta)=\lim\limits_{N\to\infty}\frac{1}{N}\mathbb{E}[\log\hat{Z}_{N,\beta}^{\omega}]
\end{equation}
is non-random. By Jensen's inequality, we have a trivial bound $p(\beta)\leq0$. It is easy to see that if $p(\beta)<0$, then $\hat{Z}_{N,\beta}^{\omega}$ decays exponentially fast and thus strong disorder holds. Therefore, we call the range of $\beta$ with $p(\beta)<0$ \textit{the very strong disorder regime}.

It has been shown in \cite[Theorem 3.2]{CY06} and \cite[Theorem 6.1]{Com07} that as $\beta$ increases, there is a phase transition from the weak disorder regime, through the strong disorder regime, to the very strong disorder regime, which we summarize in the following.
\begin{theorem}\label{t101}
There exist $0\leq\beta_{1}\leq\beta_{2}\leq\infty$, such that weak disorder holds if and only if $\beta\in\{0\}\cup(0,\beta_{1})$; strong disorder holds if and only if $\beta\in(\beta_{1},\infty)$; and very strong disorder holds if and only if $\beta\in(\beta_{2},\infty)$.
\end{theorem}
In \cite[Proposition 1.13]{Wei16}, the author showed that for the Cauchy directed polymer with $b_{n}\equiv0$ in \eqref{e102}, $\beta_{1}=0$ if and only if the random walk $S$ is recurrent. Let $\tilde{S}$ be an independent copy of $S$. Since $S-\tilde{S}$ is symmetric and thus $b_{n}\equiv0$ in \eqref{e102}, one can easily check that the recurrence of $S-\tilde{S}$ is also equivalent to $\beta_{1}=0$ by the same method used in \cite[Proposition 1.13]{Wei16}. When $\beta_{1}=0$, the model is called \textit{disorder relevant}, since for arbitrarily small $\beta>0$, disorder modifies the large scale behavior of the underlying random walk.

It is conjectured that $\beta_{1}=\beta_{2}$, i.e., the strong disorder regime coincides with the very strong disorder regime (excluding the critical $\beta$). So far, the conjecture has only been proved for the nearest-neighbor directed polymer on $\mathbb{Z}^{d+1}$ for $d=1$ in \cite{CV06} and $d=2$ in \cite{Lac10}, and for the long-range directed polymer with underlying random walks in the domain of attraction of an $\alpha$-stable law for some $\alpha\in(1,2]$ in \cite{Wei16}.

The main purpose of this paper is to show that for disorder relevant Cauchy directed polymer, under some regularity assumptions on the random walk, $\beta_{1}=\beta_{2}$, i.e., $\beta_{2}=0$. We will present the precise results in the next subsection.
\subsection{Main results}
Recall that $S$ is the random walk defined in \eqref{e101} and $\tilde{S}$ is an independent copy of $S$. Note that the expected local time of $S-\tilde{S}$ at the origin up to time $N$ is given by
\begin{equation}\label{e110}
D(N):=\sum\limits_{n=1}^{N}\mathbf{P}^{\bigotimes2}(S_{n}=\tilde{S}_{n})=\sum\limits_{n=1}^{N}\sum\limits_{x\in\mathbb{Z}}\mathbf{P}(S_{n}=x)^{2},
\end{equation}
where $\mathbf{P}^{\bigotimes2}$ is the probability on product space. The quantity $D(\cdot)$ is known as the overlap, which will be crucial in our analysis.

Note that $S-\tilde{S}$ is symmetric, and by \cite[Chapter VIII.8, Corollary]{Fel66},
\begin{equation}\label{e111}
\frac{S_{n}-\tilde{S}_{n}}{a_{n}}\overset{d}{\to}H,\quad\mbox{as}~n\to\infty,
\end{equation}
where $a_{n}$ is the same as in \eqref{e102} and $H$ is some symmetric Cauchy random variable. If $(S-\tilde{S})/h$ is an irreducible aperiodic random walk on $\mathbb{Z}$, then by Gnedenko's local limit theorem (\textit{cf.} \cite[Theorem 8.4.1]{BGT89}),
\begin{equation}\label{e112}
\mathbf{P}^{\bigotimes2}(S_{n}=\tilde{S}_{n})\sim\frac{g(0)h}{a_{n}},
\end{equation}
where $g(\cdot)$ is the density function for $H$. Hence, $S-\tilde{S}$ is recurrent if and only if $\sum_{n=1}^{\infty}a_{n}^{-1}=\infty$. Therefore, for disorder relevant Cauchy directed polymer model, the overlap $D(N)$ tends to infinity as $N$ tends to infinity.

We mention that in \cite[Proposition 3.1]{Wei16}, the author showed that
\begin{equation}\label{e113}
\sum\limits_{n=1}^{\infty}\frac{1}{nL(n)}=\infty\Leftrightarrow\sum\limits_{n=1}^{\infty}\frac{1}{a_{n}}=\infty,
\end{equation}
where $L(\cdot)$ was introduced in \eqref{e101}. When an explicit close form for $a_{n}$ is hard to deduce, \eqref{e113} provides an alternative way for checking the recurrence of $S-\tilde{S}$.

To prove $\beta_{2}=\beta_{1}=0$, we need an extra assumption on the distribution of $S$, which is
\begin{equation}\label{e114}
\begin{split}
\mathbf{P}(S_{1}=k)\sim pL(k)k^{-2},\quad\mbox{as}~k\to\infty,\\
\mathbf{P}(S_{1}=-k)\sim qL(k)k^{-2},\quad\mbox{as}~k\to\infty.
\end{split}
\end{equation}
By \cite[Proposition 1.5.10]{BGT89}, the stronger regular condition \eqref{e114} implies \eqref{e101}. The reason that we assume \eqref{e114} is that we want to have a better control of the local behavior of $S$. The following result will be used in our proof.
\begin{theorem}[{\hspace{1sp}\cite[Theorem 2.4]{Ber17}}]\label{t102}
Let $S$ be a random walk that satisfies \eqref{e114}, and $a_{n}$ and $b_{n}$ be the constants in \eqref{e102}. Then there exist positive constants $c_{1}$ and $c_{2}$, such that for any $|k|\geq a_{n}$ with $\mathbf{P}(S_{n}-\lfloor b_{n}\rfloor=k)\neq0$,
\begin{equation}\label{e115}
c_{1}nL(|k|)k^{-2}\leq\mathbf{P}(S_{n}-\lfloor b_{n}\rfloor=k)\leq c_{2}nL(|k|)k^{-2}.
\end{equation}
\end{theorem}
\begin{remark}\label{r101}
Although only the upper bound for $\mathbf{P}(S_{n}-\lfloor b_{n}\rfloor=|k|)$ was presented in \cite[Theorem 2.4]{Ber17}, the author also showed that if $|k|/a_{n}\to\infty$ as $n\to\infty$, then
\begin{equation}\label{e116}
\mathbf{P}(S_{n}-\lfloor b_{n}\rfloor=|k|)\sim(p\mathbbm{1}_{k>0}+q\mathbbm{1}_{k<0})nL(|k|)k^{-2},\quad\mbox{as}~n\to\infty.
\end{equation}
One may check that the lower bound in \eqref{e115} can be proved by the method developed in proving \eqref{e116}. Note that by \eqref{e103},
\begin{equation}\label{e117}
\frac{nL(|k|)}{k^{2}}\sim\frac{a_{n}L(|k|)}{k^{2}L(a_{n})},\quad\mbox{as}~n\to\infty, 
\end{equation}
which will be useful later.
\end{remark}
Now we are ready to present our main results. Recall $\beta_{1}$ and $\beta_{2}$ from Theorem \ref{t101}. Throughout the rest of this paper, we assume $\beta_{1}=0$, i.e., the model is disorder relevant, which is an equivalent condition for $S-\tilde{S}$ to be recurrent according to the statement right below Theorem \ref{t101}.

We first show that with some extra assumption on the underlying random walk $S$, $\beta_{2}=\beta_{2}=0$, i.e, the free energy is strictly negative as soon as $\beta>0$.
\begin{theorem}\label{t103}
Let the Cauchy directed polymer model be defined as in Subsection \ref{s101}. We assume that the underlying random walk $S$ satisfies \eqref{e114} and $S-\tilde{S}$ is recurrent. We set
\begin{equation}\label{e118}
D^{-1}(x):=\max\{N: D(N)\leq x\}.
\end{equation}
If the centering constant $b_{n}\equiv0$ in \eqref{e102}, then for arbitrarily small $\epsilon>0$, there exists a $\beta^{(1)}>0$, such that for any $\beta\in(0,\beta^{(1)})$,
\begin{equation}\label{e119}
p(\beta)\leq-(D^{-1}((1+\epsilon)\beta^{-2}))^{-(1+\epsilon)}.
\end{equation}
\end{theorem}
If we drop the assumption $b_{n}\equiv0$, then some technical difficulties will arise. We will elaborate on this point when we prove Theorem \ref{t103}.

We can also give a lower bound for the free energy, and the lower bound is valid under fairly general conditions.
\begin{theorem}\label{t104}
Let the Cauchy directed polymer model be defined as in Subsection \ref{s101}. If $S-\tilde{S}$ is recurrent, then for arbitrarily small $\epsilon>0$, there exists a $\beta^{(2)}>0$, such that for $\beta\in(0,\beta^{(2)})$,
\begin{equation}\label{e120}
p(\beta)\geq-D^{-1}((1-\epsilon)\beta^{-2})^{-(1-\epsilon)}.
\end{equation}
\end{theorem}
Note that in Theorem \ref{t104}, the underlying random walk $S$ only needs to satisfy \eqref{e101}. Neither \eqref{e114} nor $b_{n}\equiv0$ are needed.

In particular, if $S$ satisfies \eqref{e114} with the slowly varying function $L(\cdot)\equiv c$, then $a_{n}$ can be chosen to be $c(p+q)n$. Hence, $D(N)\sim\log{N}/c(p+q)$ and $D^{-1}(x)\sim\exp(c(p+q)x)$. Since $S-\tilde{S}$ is recurrent due to \eqref{e113}, we have:
\begin{corollary}\label{c105}
Let the Cauchy directed polymer model be defined as in Subsection \ref{s101}. If the underlying random walk $S$ satisfies \eqref{e114} with $L(\cdot)\equiv c$ and the centering constant $b_{n}\equiv 0$ in \eqref{e102}, then by Theorem \ref{t103} and Theorem \ref{t104},
\begin{equation}
\lim\limits_{\beta\to0}\beta^{2}\log(-p(\beta))=-c(p+q).
\end{equation}
\end{corollary}
\subsection{Organization and discussion}
Theorem \ref{t103} will be proved in Section \ref{s200} by a now classic fractional-moment/coarse-graining/change-of-measure procedure. We will adapt the approaches developed in \cite{BL16,BL17}.

Theorem \ref{t104} will be proved in Section \ref{s300} using a second moment computation introduced in \cite{BL17} and a concentration inequality developed in \cite{CTT15}.

Although our proof techniques are adaptation of known methods, some new subtle arguments are needed, since the random walk in the Cauchy domain of attraction is much harder to deal with than $2$-dimensional simple random walk.

We believe that the approach in this paper can be applied to handle the 2 dimensional long-range directed polymer with stable exponent $\alpha=2$, which is the critical case for long-range directed polymer on $\mathbb{Z}^{2+1}$. With some regularity assumption on the underlying random walk $S$, one can prove $\beta_{2}=\beta_{1}=0$ if $S-\tilde{S}$ is recurrent by our methods. 

It dose not seem likely that one can prove the upper bound \eqref{e119} under the general condition in Theorem \ref{t104} by the fractional-moment/coarse-graining/change-of-measure procedure. One may have to find a totally new approach to deal with the upper bound in the general case.
\section{Proof of Theorem \ref{t103}}\label{s200}
We start with the fractional-moment method. Recall \eqref{e109}, for any $\theta\in(0,1)$,
\begin{equation*}
p(\beta)=\lim\limits_{N\to\infty}\frac{1}{N}\mathbb{E}[\log\hat{Z}_{N,\beta}^{\omega}]\leq\varliminf\limits_{N\to\infty}\frac{1}{\theta N}\log\mathbb{E}[(\hat{Z}_{N,\beta}^{\omega})^{\theta}]
\end{equation*}
by Jensen's inequality. In this proof, $\theta$ cannot be chosen arbitrarily. In fact, we will see later that $\theta$ should be larger than $1/2$. Then our strategy is to chose a coarse-graining length $l=l(\beta)$, write $N=ml$, and let $m$ tend to infinity. Along the subsequence $N=ml$, we have
\begin{equation*}
p(\beta)\leq\varliminf\limits_{m\to\infty}\frac{1}{ml\theta}\log\mathbb{E}[(\hat{Z}_{ml,\beta}^{\omega})^{\theta}].
\end{equation*}
If we can prove
\begin{equation}\label{e201}
\mathbb{E}[(\hat{Z}_{ml,\beta}^{\omega})^{\theta}]\leq2^{-m},
\end{equation}
then we obtain $p(\beta)<0$. In order to further prove the upper bound \eqref{e117} for any $\epsilon>0$, one appropriate choice of $l$ is
\begin{equation}\label{e202}
l=l(\beta):=\inf\{n\in\mathbb{N}: D(\lfloor n^{1-\epsilon^{2}}\rfloor)\geq(1+\epsilon)\beta^{-2}\}.
\end{equation}
Note that $D(N)$ tends to infinity as $N$ tends to infinity, since $S-\tilde{S}$ is recurrent. Thus, $l$ tends to infinity as $\beta$ tends to $0$.

Now we introduce the coarse-graining method. First, we partition all real number $\mathbb{R}$ into blocks of size $a_{l}$ by setting
\begin{equation*}
I_{y}:=ya_{l}+(-a_{l}/2, a_{l}/2],\quad\forall y\in\mathbb{Z},
\end{equation*}
where $a_{l}$ is the scaling constant in \eqref{e102}. Since $a_{l}$ tends to infinity as $l$ tends to infinity, we can choose $a_{l}$ to be an integer and thus $ya_{l}$ is also an integer. Note that $(I_{y})_{y\in\mathbb{Z}}$ is a disjoint family and $\cup_{y\in\mathbbm{Z}}I_{y}=\mathbbm{R}$. Next, for any $\mathcal{Y}=(y_{1},\cdots,y_{m})$, define
\begin{equation*}
\mathcal{T}_{\mathcal{Y}}=\{S_{il}\in I_{y_{i}}, \mbox{for}~1\leq i\leq m\},
\end{equation*}
and we say $\mathcal{Y}$ is a coarse-grained trajectory for $S\in\mathcal{T}_{\mathcal{Y}}$. We can now decompose the partition function $\hat{Z}_{ml,\beta}^{\omega}$ in terms of different coarse-grained trajectories by
\begin{equation*}
\hat{Z}_{ml,\beta}^{\omega}=\sum\limits_{\mathcal{Y}\in\mathbb{Z}^{m}}\mathbf{E}\left[\exp\left(\sum\limits_{n=1}^{ml}(\beta\omega_{n,S_{n}}-\lambda(\beta))\right)\mathbbm{1}_{\{S\in\mathcal{T}_{\mathcal{Y}}\}}\right]:=\sum\limits_{\mathcal{Y}\in\mathbb{Z}^{m}}Z_{\mathcal{Y}}.
\end{equation*}
By the inequality $(\sum_{n}a_{n})^{\theta}\leq\sum_{n}a_{n}^{\theta}$ for positive sequence $(a_{n})_{n}$ and $\theta\in(0,1]$,
\begin{equation}\label{e203}
\mathbb{E}[(\hat{Z}_{ml,\beta}^{\omega})^{\theta}]\leq\sum\limits_{\mathcal{Y}\in\mathbb{Z}^{m}}\mathbb{E}[(Z_{\mathcal{Y}})^{\theta}].
\end{equation}
Therefore, to prove \eqref{e201}, we only need to prove
\begin{proposition}\label{p201}
If $l$ is sufficiently large, then uniformly in $m\in\mathbb{N}$, we have
\begin{equation*}
\sum\limits_{\mathcal{Y}\in\mathbb{Z}^{m}}\mathbb{E}[(Z_{\mathcal{Y}})^{\theta}]\leq2^{-m}.
\end{equation*}
\end{proposition}
To prove Proposition \ref{p201}, we need a change-of-measure argument. For any $\mathcal{Y}\in\mathbb{Z}^{m}$, we introduce a positive function $g_{\mathcal{Y}}(\omega)$, which can be considered as a probability density after scaling. Then by H\"{o}lder's inequality,
\begin{equation}\label{e204}
\mathbb{E}[(Z_{\mathcal{Y}})^{\theta}]=\mathbb{E}[g_{\mathcal{Y}}^{-\theta}(g_{\mathcal{Y}}Z_{\mathcal{Y}})^{\theta}]\leq\left(\mathbb{E}[g_{\mathcal{Y}}^{-\theta/(1-\theta)}]\right)^{1-\theta}\left(\mathbb{E}[g_{\mathcal{Y}}Z_{\mathcal{Y}}]\right)^{\theta}.
\end{equation}
Here $\mathcal{M}_{g_{\mathcal{Y}}}(\cdot):=\mathbb{E}[g_{\mathcal{Y}}\mathbbm{1}_{(\cdot)}]$ can be considered as a new measure. We will choose $g_{\mathcal{Y}}$ such that the expected value of $Z_{\mathcal{Y}}$ under $\mathcal{M}_{g_\mathcal{Y}}$ is significantly smaller than that under the original measure $\mathbb{E}$, and the cost of change of measure, the term $\mathbb{E}[g_{\mathcal{Y}}^{-\theta/(1-\theta)}]$, is not too large.

To choose $g_{\mathcal{Y}}$, we need to first introduce some notation. We can first choose an integer $R$ (not dependent on $\beta$) and then define space-time blocks (with the convention $y_{0}=0$)
\begin{equation*}
B_{i,y_{i-1}}:=[(i-1)l+1,\cdots,il]\times\tilde{I}_{y_{i-1}},~\mbox{for}~i=1,\cdots,m,
\end{equation*}
where
\begin{equation}\label{e205}
\tilde{I}_{y}=ya_{l}+(-Ra_{l},Ra_{l}).
\end{equation}
Since $S$ is in the domain of attraction of a 1-stable L\'{e}vy process, the graph of $(S_{(i-1)l+k})_{k=1}^{l}$ with $S_{(i-1)l}=y_{i-1}$ is contained within $B_{i,y_{i-1}}$ with probability close to 1 when $R$ is large enough. Therefore, it suffices to perform the change of measure on $\omega$ in $B=\cup_{i=1}^{m}B_{i,y_{i-1}}$. By translation invariance, it is natural to choose
\begin{equation*}
g_{\mathcal{Y}}(\omega)=\prod\limits_{i=1}^{m}g_{i,y_{i-1}}(\omega)
\end{equation*}
such that each $g_{i,y_{i-1}}$ depends only on $\omega$ in $B_{i,y_{i-1}}$.

To make $\mathbb{E}[g_{\mathcal{Y}}Z_{\mathcal{Y}}]$ small, we can construct $g_{\mathcal{Y}}$ according to the following heuristics. we first set a threshold. For any block $B_{i,y}$, If the contribution of $\omega$ in $B_{i, y}$ to the partition function exceeds the threshold, then we choose $g_{i,y}$ to be small. If the contribution of $\omega$ in $B_{i, y}$ to the partition function is less than the threshold, then we simply set $g_{i,y}$ to be $1$. Before we present the exact construction of $g_{\mathcal{Y}}$, we need to define some auxiliary quantities, which will help us compute the contribution to $Z_{N,\beta}^{\omega}$ from each block $B_{i,y}$.

For arbitrarily small $\epsilon>0$, we introduce
\begin{equation}\label{e206}
u=u(l):=\lfloor l^{1-\epsilon^{2}}\rfloor\quad\mbox{and}\quad q=q(l):=\frac{1}{\epsilon^{2}}\max\left\{\log\left(\sqrt{\varphi(l)}\right),\log D(l)\right\},
\end{equation}
where $\varphi(\cdot)$ is the slowly varying function in \eqref{e103}. Note that by \eqref{e202}, $u$ and $q$ both tend to infinity as $\beta$ tends to $0$, and the definitions of $q$ and $u$ ensure that
\begin{equation}\label{e207}
q\ll u\ll l\quad\mbox{and}\quad1+\epsilon\leq\beta^{2}D(u)\leq1+2\epsilon.
\end{equation}
We will use \eqref{e207} repeatedly.

Then we define $X(\omega)$ depending on $\omega$ in $B_{1,0}$ by
\begin{equation}\label{e208}
X(\omega):=\frac{1}{\sqrt{2Rla_{l}}D(u)^{q/2}}\sum\limits_{\underline{x}\in(\tilde{I}_{0})^{q+1},\underline{t}\in J_{l,u}}\mathbf{P}(\underline{t},\underline{x})\omega_{\underline{t},\underline{x}}
\end{equation}
with
\begin{align*}
&\underline{x}:=(x_{0},\cdots,x_{q})\quad\mbox{and}\quad\underline{t}:=(t_{0},\cdots,t_{q}),\\
&J_{l,u}:=\{\underline{t}: 1\leq t_{0}<\cdots<t_{q}\leq l, t_{i}-t_{i-1}\leq u, \forall j=1,\cdots,q\},
\end{align*}
\begin{equation}\label{e209}
\mathbf{P}(\underline{t},\underline{x}):=\prod\limits_{i=1}^{q}\mathbf{P}(S_{t_{i}}-S_{t_{i-1}}=x_{i}-x_{i-1}),
\end{equation}
and
\begin{equation*}
\omega_{\underline{t},\underline{x}}:=\prod\limits_{i=0}^{q}\omega_{t_{i},x_{i}}，
\end{equation*}
where the constant $R$ is chosen to be the same as in \eqref{e205}.

We can regard $X(\omega)$ as an approximation of the contribution from $\omega$ in $B_{1,0}$ to the normalized partition function $\hat{Z}_{N,\beta}^{\omega}$. It can be viewed as something like the $q$-th order term in the Taylor expansion of $\hat{Z}_{N,\beta}^{\omega}$ in $\omega$. We introduce this approximation since $X(\omega)$ is a mutilinear combination of $\omega_{t,x}$'s, which is treatable, while it is rarely possible to do computation on the partition function directly. One may refer to \cite[Section 4.2]{BL16} for more discussions concerning the choice of $X(\omega)$.

It is not hard to check that by \eqref{e107} and \eqref{e110},
\begin{equation}\label{e210}
\mathbb{E}[X(\omega)]=0\quad\mbox{and}\quad\mathbb{E}[(X(\omega))^{2}]\leq1.
\end{equation}
Then, by translation invariance, for the contribution from $\omega$ in any block $B_{i,y}$, we can define
\begin{equation}\label{e211}
X^{(i,y)}(\omega):=X(\theta_{l}^{i-1,y}\omega),
\end{equation}
where $\theta_{l}^{i-1,y}\omega_{j,x}:=\omega_{j+(i-1)l,x+ya_{l}}$ is a shift operator.

Now we can set
\begin{equation}\label{e212}
g_{i,y}(\omega):=\exp\left(-K\mathbbm{1}_{\{X^{(i,y)}(\omega)\geq\exp(K^{2})\}}\right),
\end{equation}
where $K$ is a fixed constant independent of any other parameter. We then have
\begin{equation*}
\mathbb{E}[(g_{i,y})^{-\theta/(1-\theta)}]=1+(\exp(\theta K/(1-\theta))-1)\mathbb{P}(X^{(i,y)}(\omega)\geq\exp(K^{2}))\leq2
\end{equation*}
by Chebyshev's inequality and \eqref{e210} if we choose $K$ large enough. Since $g_{i,y_{i-1}}$ and $g_{j,y_{j-1}}$ are defined on disjoint blocks $B_{i,y_{i-1}}$ and $B_{j,y_{j-1}}$ for $i\neq j$, by independence of $\omega$ in $B_{i,y_{i-1}}$ and $B_{j,y_{j-1}}$,
\begin{equation}\label{e213}
\left(\mathbb{E}[g_{\mathcal{Y}}^{-\theta/(1-\theta)}]\right)^{1-\theta}=\left(\prod\limits_{i=1}^{m}\mathbb{E}[g_{i,y_{i-1}}^{-\theta/(1-\theta)}]\right)^{1-\theta}\leq2^{m(1-\theta)}\leq2^{m}.
\end{equation}

Next, we turn to analyze $\mathbb{E}[g_{\mathcal{Y}}Z_{\mathcal{Y}}]$ in \eqref{e204}. We can rewrite it as
\begin{equation}\label{e214}
\mathbb{E}[g_{\mathcal{Y}}Z_{\mathcal{Y}}]=\mathbf{E}\left[\mathbb{E}\left[g_{\mathcal{Y}}\exp\left(\sum\limits_{n=1}^{ml}(\beta\omega_{n,S_{n}}-\lambda(\beta))\right)\right]\mathbbm{1}_{\{S\in\mathcal{T}_{\mathcal{Y}}\}}\right].
\end{equation}
For any given trajectory of $S$, we define a change of measure by
\begin{equation*}
\frac{\mbox{d}\mathbb{P}^{S}}{\mbox{d}\mathbb{P}}(\omega):=\exp\left(\sum\limits_{n=1}^{ml}(\beta\omega_{n,S_{n}}-\lambda(\beta))\right).
\end{equation*}
We can check that $\mathbb{P}^{S}$ is a probability measure, and $\omega$ remains a family of independent random variables under $\mathbb{P}^{S}$, but the distribution of $\omega_{n,S_{n}}$ is exponentially tilted with
\begin{equation}\label{e215}
\mathbb{E}^{S}[\omega_{n,x}]=\lambda'(\beta)\mathbbm{1}_{\{S_{n}=x\}}\quad\mbox{and}\quad\mathbb{V}\mbox{ar}^{S}(\omega_{n,x})=1+(\lambda''(\beta)-1)\mathbbm{1}_{\{S_{n}=x\}}.
\end{equation}
One can check that
\begin{equation*}
\lim\limits_{\beta\to0}\frac{\lambda'(\beta)}{\beta}=1\quad\mbox{and}\quad\lim\limits_{\beta\to0}\lambda''(\beta)=1.
\end{equation*}
Hence, for $\epsilon$ given in Theorem \ref{t103}, when $\beta$ is sufficiently small, we have
\begin{equation}\label{e216}
\left|\frac{\lambda'(\beta)}{\beta}-1\right|\leq\epsilon^{3}\quad\mbox{and}\quad|\lambda''(\beta)-1|\leq\frac{\epsilon^{3}}{2}.
\end{equation}
By independence of $\omega$, \eqref{e214} can be further rewritten as
\begin{equation}\label{e217}
\mathbb{E}[g_{\mathcal{Y}}Z_{\mathcal{Y}}]=\mathbf{E}\left[\mathbb{E}^{S}[g_{\mathcal{Y}}] \mathbbm{1}_{\{S\in\mathcal{T}_{\mathcal{Y}}\}}\right]=\mathbf{E}\left[\prod\limits_{i=1}^{m}\mathbb{E}^{S}[g_{i,y_{i-1}}]\mathbbm{1}_{\{S_{il}\in I_{y_{i}}\}}\right].
\end{equation}
Applying the Markov property by consecutively conditioning on $S_{(m-1)l}, S_{(m-2)l}, \cdots$ and taking maximum according to $x\in I_{y_{i}-1}$ each time, \eqref{e217} can be bounded above by
\begin{equation*}
\prod\limits_{i=1}^{m}\max\limits_{x\in I_{y_{i-1}}}\mathbf{E}\left[\mathbb{E}^{S}[g_{i,y_{i-1}}]\mathbbm{1}_{\{S_{il}\in I_{y_{i}}\}}\bigg|S_{(i-1)l}=x\right].
\end{equation*}
Using translation invariance \eqref{e211} and noting that $f(y_{1},y_{2},\cdots,y_{m})=(y_{1}, y_{2}-y_{1},\cdots, y_{m}-y_{m-1})$ is a bijection from $\mathbb{Z}^{m}$ to $\mathbb{Z}^{m}$, we sum $(\mathbb{E}[g_{\mathcal{Y}}Z_{\mathcal{Y}}])^{\theta}$ over $\mathcal{Y}\in\mathbb{Z}^{m}$ and than obtain
\begin{equation}\label{e218}
\sum\limits_{\mathcal{Y}\in\mathbb{Z}^{m}}\left(\mathbb{E}[g_{\mathcal{Y}}Z_{\mathcal{Y}}]\right)^{\theta}\leq\left(\sum\limits_{y\in\mathbb{Z}}\max\limits_{x\in I_{0}}\left(\mathbf{E}^{x}\left[\mathbb{E}^{S}[g_{1,0}]\mathbbm{1}_{\{S_{l}\in I_{y}\}}\right]\right)^{\theta}\right)^{m},
\end{equation}
where $\mathbf{E}^{x}$ is the expectation with respect to $\mathbf{P}^{x}$, which is the probability measure for random walk $S$ starting at $x$.
\begin{remark}\label{r201}
Here we explain why we have to assume $b_{n}\equiv0$ in \eqref{e102}. For the coarse-grained trajectory of $S$, $S_{il}-S_{(i-1)l}-b_{l}$ should be of scale $a_{l}$. However, if $\mathbf{E}[S_{1}]$ does not exist, then $b_{n}$ may not be proportional to $n$. Hence, for $k>n$,
\begin{equation*}
(S_{kl}-b_{kl})-(S_{nl}-b_{nl})\overset{d}{=}S_{(k-n)l}-(b_{kl}-b_{nl})\overset{d}{\neq}S_{(k-n)l}-b_{(k-n)l},
\end{equation*}
which will cause the subsequent use of the Markov property to fail.

If $b_{n}$ is proportional to $n$, then when handling the coarse-grained trajectory and defining \eqref{e209}, we can replace all $S_{n}$ by $S_{n}-b_{n}$ such that all of our arguments are still valid. Nevertheless, we just assume $b_{n}\equiv0$ for simplicity.
\end{remark}
Now by \eqref{e204}, \eqref{e213} and \eqref{e218}, to prove Proposition \ref{p201}, we only need to show
\begin{proposition}\label{p202}
For small enough $\beta>0$,
\begin{equation}\label{e219}
\sum\limits_{y\in\mathbb{Z}}\max\limits_{x\in I_{0}}\left(\mathbf{E}^{x}\left[\mathbb{E}^{S}[g_{1,0}]\mathbbm{1}_{\{S_{l}\in I_{y}\}}\right]\right)^{\theta}\leq\frac{1}{4}.
\end{equation}
\end{proposition}
To prove Proposition \ref{p202}, we split the summation in \eqref{e219} into two parts.

Firstly, since $g_{1,0}\leq1$,
\begin{equation}\label{e220}
\sum\limits_{|y|\geq M}\max\limits_{x\in I_{0}}\left(\mathbf{E}^{x}\left[\mathbb{E}^{S}[g_{1,0}]\mathbbm{1}_{\{S_{l}\in I_{y}\}}\right]\right)^{\theta}\leq\sum\limits_{|y|\geq M}\max\limits_{x\in I_{0}}\mathbf{P}^{x}(S_{l}\in I_{y})^{\theta}.
\end{equation}
By Theorem \ref{t102}, when $M$ is large enough and fixed, for any $k\geq M$ and $j\in\{1,\cdots,a_{l}-1\}$,
\begin{equation*}
\mathbf{P}(S_{l}=ka_{l}+j)\leq C\frac{a_{l}L(ka_{l}+j)}{(ka_{l}+j)^{2}L(a_{l})}\leq C \frac{L(ka_{l}+j)}{k^{2}a_{l}L(a_{l})}.
\end{equation*}
Then by Potter bounds (\textit{cf}. \cite[Theorem 1.5.6]{BGT89}), for any $\gamma>0$, there exist some constant $C$, such that for $k$ and $j$,
\begin{equation*}
\frac{L(ka_{l}+j)}{L(a_{l})}\leq Ck^{\gamma}\quad\mbox{uniformly}.
\end{equation*}
Hence, the summand in \eqref{e220} can be uniformly bounded from above by $Ck^{\theta(\gamma-2)}$. Therefore, when $\gamma<1$, we can choose $\theta$ close to $1$ enough such that $\theta(\gamma-2)<-1$ and then \eqref{e220} can be bounded from above by $1/8$ for sufficiently large $M$.

Next, we turn to the control of the summand in \eqref{e219} for $|y|\leq M$. We can first apply a trivial bound
\begin{equation}\label{e221}
\mathbf{E}^{x}\left[\mathbb{E}^{S}[g_{1,0}]\mathbbm{1}_{\{S_{l}\in I_{y}\}}\right]\leq\mathbf{E}^{x}\left[\mathbb{E}^{S}[g_{1,0}]\right].
\end{equation}
Then we want to show
\begin{lemma}\label{l203}
For any $\eta>0$, we can choose $K$ large enough in \eqref{e212}, which only depends on $\eta$, such that for small enough $\beta>0$, we have
\begin{equation*}
\max\limits_{x\in I_{0}}\mathbf{E}^{x}\left[\mathbb{E}^{S}[g_{1,0}]\right]\leq\eta.
\end{equation*}
\end{lemma}
By \eqref{e221} and Lemma \ref{l203}, if we choose $\eta=(16M)^{-1/\theta}$, then
\begin{equation}\label{e222}
\sum\limits_{|y|\leq M}\max\limits_{x\in I_{0}}\left(\mathbf{E}^{x}\left[\mathbb{E}^{S}[g_{1,0}]\mathbbm{1}_{\{S_{l}\in I_{y}\}}\right]\right)^{\theta}\leq\frac{1}{8}.
\end{equation}
Combining \eqref{e222} and the upper bound for \eqref{e220}, we deduce Proposition \ref{p202}. Therefore, it only remains to prove Lemma \ref{l203}.

Indeed, Lemma \ref{l203} follows from the following two lemmas.
\begin{lemma}\label{l204}
For any $\delta>0$, we can choose a large enough $R$ in \eqref{e205}, which only depends on $\delta$ and the $\epsilon$ in Theorem \ref{t102}, such that for small enough $\beta>0$, and for any $x\in I_{0}$, we have
\begin{equation*}
\mathbf{P}^{x}(\mathbb{E}^{S}[X]\geq(1+\epsilon^{2})^{q})\geq1-\delta.
\end{equation*}
\end{lemma}
\begin{lemma}\label{l205}
If $\beta$ is positive and sufficiently small, then for any trajectory $S$ of the underlying random walk, we have
\begin{equation*}
\mathbb{V}\mbox{\rm{ar}}^{S}(X)\leq(1+\epsilon^{3})^{q}.
\end{equation*}
\end{lemma}
We postpone the proof of Lemma \ref{l204} and Lemma \ref{l205}, and deduce Lemma \ref{l203} first.
\begin{proof}[Proof of Lemma \ref{l203}]
By the definition of $g_{1,0}$, for any trajectory $S$, we have the following trivial bound
\begin{equation}\label{e223}
\mathbb{E}^{S}[g_{1,0}]\leq\exp(-K)+\mathbb{P}^{S}(X(\omega)\leq\exp(K^{2})).
\end{equation}
By Chebyshev's inequality,
\begin{equation}\label{e224}
\mathbb{P}^{S}(X(\omega)\leq\exp(K^{2}))\leq(\exp(K^{2})-\mathbb{E}^{S}[X])^{-2}\mathbb{V}\mbox{ar}^{S}(X).
\end{equation}
We denote $A=\{\mathbb{E}^{S}[X]\geq(1+\epsilon^{2})^{q}\}$. For any $x\in I_{0}$, by \eqref{e224}, Lemma \ref{l204} and Lemma \ref{l205}, we then have
\begin{equation}\label{e225}
\begin{split}
&\mathbf{E}^{x}\left[\mathbb{P}^{S}(X(\omega)\leq\exp(K^{2}))\right]\\
\leq&\mathbf{P}^{x}(A^{c})+\mathbf{E}^{x}\left[\mathbb{P}^{S}(X(\omega)\leq\exp(K^{2}))\mathbbm{1}_{A}\right]\\
\leq&\delta+\frac{(1+\epsilon^{3})^{q}}{2(1+\epsilon^{2})^{2q}},
\end{split}
\end{equation}
where we use the fact that $(1+\epsilon^{2})^{q}-\exp(K^{2})\geq\sqrt{2}(1+\epsilon^{2})^{q}$ to obtain the last line, since $q$ can be made arbitrarily large by choosing $\beta$ close enough to $0$.

Now we first take $\mathbf{E}^{x}$-expectation on the both sides of \eqref{e223}. Then, we choose $K$ large enough such that $\exp(-K)<\eta/3$. Next, we let $\beta$ tend to $0$ so that the last line of \eqref{e225} is smaller than $2\eta/3$, which implies Lemma \ref{l203}.
\end{proof}
The proofs of Lemma \ref{l204} and Lemma \ref{l205} involve some long and tedious computations. Hence, we put each proof in one subsection to make the structure more clear and we will write some intermediate steps as lemmas to clarify the proofs.
\subsection{Proof of Lemma \ref{l204}}In this subsection, we prove Lemma \ref{l204}.
\begin{proof}[Proof of Lemma \ref{l204}]
First, we recall the definition \eqref{e208} of $X$. Note that $\omega$ is a family of independent random variables under $\mathbb{P}^{S}$, and by \eqref{e215}, $\mathbb{E}^{S}[\omega_{n,x}]=0$ if $S_{n}\neq x$. Hence, for any trajectory of $S$, we have
\begin{equation}\label{e226}
\begin{split}
\mathbb{E}^{S}[X]=&\frac{(\lambda'(\beta))^{q+1}}{\sqrt{2Rla_{l}}D(u)^{q/2}}\sum\limits_{\underline{t}\in J_{l,u}}\mathbf{P}(\underline{t},\underline{S}^{(\underline{t})})\mathbbm{1}_{\{S_{t_{k}}\in\tilde{I}_{0},\forall k\in\{0,\cdots,q\}\}}\\
\geq&\frac{(\lambda'(\beta))^{q+1}}{\sqrt{2Rla_{l}}D(u)^{q/2}}\sum\limits_{\underline{t}\in J_{l,u}}\mathbf{P}(\underline{t},\underline{S}^{(\underline{t})})\mathbbm{1}_{\left\{\max\limits_{1\leq t\leq l}|S_{t}|\leq Ra_{l}\right\}},
\end{split}
\end{equation}
where
\begin{equation}\label{e227}
\underline{S}^{(\underline{t})}:=(S_{t_{0}},\cdots,S_{t_{q}})
\end{equation}
and we will use notation \eqref{e227} in what follows. We emphasize that in \eqref{e226}, the trajectory $\underline{S}^{(\underline{t})}$ should be substituted into the $\underline{x}$ in \eqref{e209} and readers should not mix it up with the random walk $S$ in \eqref{e209}.

Note that for any $x\in I_{0}=(-a_{l}/2, a_{l}/2]$,
\begin{equation}\label{e228}
\mathbf{P}^{x}\left(\max\limits_{1\leq t\leq l}|S_{t}|>Ra_{l}\right)\leq\mathbf{P}\left(\max\limits_{1\leq t\leq l}|S_{t}|>(R-1)a_{l}\right).
\end{equation}
Since $S$ is attracted to some $1$-stable L\'{e}vy process, for any $\delta>0$, we can choose $R=R(\delta,\epsilon)$ large enough such that uniformly in $l$, the probability in \eqref{e228} is smaller than $\delta/2$. In what follows, we will simply write $R$ for $R(\delta,\epsilon)$.

On the event $\{\max_{1\leq t\leq l}|S_{t}|\leq Ra_{l}\}$, by \eqref{e216}, \eqref{e206} and \eqref{e207}, we have
\begin{equation}\label{e229}
\begin{split}
\mathbb{E}^{S}[X]&\geq\frac{\beta}{\sqrt{2R\varphi(l)}}(1-\epsilon^{3})^{q+1}(\beta^{2}D(u))^{q/2}\frac{1}{lD(u)^{q}}\sum\limits_{\underline{t}\in J_{l,u}}\mathbf{P}(\underline{t},\underline{S}^{(\underline{t})})\\
&\geq\frac{\beta}{\sqrt{2R}}\frac{(1-\epsilon^{3})^{q+1}(1+\epsilon)^{q/2}}{\exp(\epsilon^{2}q)}\frac{1}{lD(u)^{q}}\sum\limits_{\underline{t}\in J_{l,u}}\mathbf{P}(\underline{t},\underline{S}^{(\underline{t})}).
\end{split}
\end{equation}
Note that for $\epsilon$ small enough, by \eqref{e207},
\begin{equation*}
\begin{split}
&\beta\frac{(1-\epsilon^{3})^{q+1}(1+\epsilon)^{q/2}}{(1+\epsilon^{2})^{2q}\exp(\epsilon^{2}q)}\geq\beta\left(1+\frac{\epsilon}{20}\right)^{q}\\
\geq&\beta\left(1+\frac{\epsilon}{20}\right)^{\frac{1}{\epsilon^{2}}\log{D(l)}}\gg\beta\exp(\log{D(u)})\geq\frac{1+\epsilon}{\beta}\gg1.
\end{split}
\end{equation*}
Hence,
\begin{equation*}
\frac{\beta}{\sqrt{2R}}\frac{(1-\epsilon^{3})^{q+1}(1+\epsilon)^{q/2}}{\exp(\epsilon^{2}q)}\geq(1+\epsilon^{2})^{2q}
\end{equation*}
and \eqref{e229} implies that
\begin{equation}\label{e230}
\mathbb{E}^{S}[X]\geq(1+\epsilon^{2})^{2q}\frac{1}{lD(u)^{q}}\sum\limits_{t\in J_{l,u}}\mathbf{P}(\underline{t},\underline{S}^{(\underline{t})}).
\end{equation}
Recall that the probability in \eqref{e228} is smaller than $\delta/2$ and by \eqref{e230} on $\{\max_{1\leq t\leq l}|S_{t}|\leq Ra_{l}\}$, we have
\begin{equation}\label{e231}
\mathbf{P}^{x}\left(\mathbb{E}^{S}[X]<(1+\epsilon^{2})^{q}\right)\leq\frac{\delta}{2}+\mathbf{P}^{x}\left(\frac{1}{lD(u)^{q}}\sum\limits_{t\in J_{l,u}}\mathbf{P}(\underline{t},\underline{S}^{(\underline{t})})<\frac{1}{(1+\epsilon^{2})^{q}}\right).
\end{equation}
To bound the probability on the right-hand side of \eqref{e231}, we introduce a random variable
\begin{equation*}
W_{l}=W_{l}(S):=\frac{1}{lD(u)^{q}}\sum\limits_{\underline{t}\in J'_{l,u}}\mathbf{P}(\underline{t},\underline{S}^{(\underline{t})}),
\end{equation*}
where
\begin{equation*}
J'_{l,u}=\{\underline{t}\in J_{l,u}: 1\leq t_{0}\leq l/2\}.
\end{equation*}
Since $J'_{l,u}\subset J_{l,u}$, it suffices to prove
\begin{equation}\label{e232}
\mathbf{P}^{x}\left(W_{l}<\frac{1}{(1+\epsilon^{2})^{q}}\right)\leq\frac{\delta}{2}.
\end{equation}
Note that by the definition of $\mathbf{P}(\underline{t},\underline{S}^{(\underline{t})})$, the law of $W_{l}$ does not depend on the starting point $S_{0}=x$. Hence, during the rest of the proof, we can simply use $\mathbf{P}$ instead of $\mathbf{P}^{x}$ for short. Our strategy to prove \eqref{e232} is to show that the mean of $W_{l}$ is $1/2$ and the variance of $W_{l}$ can be controlled.

First, by recalling the definition of $l,u,$ and $q$, when $\beta$ is small enough, $l/2+qu<l$. Since the value of $\mathbf{P}(\underline{t},\underline{S}^{(\underline{t})})$ does not depend on $S_{t_{0}}$, we have
\begin{equation}\label{e233}
\begin{split}
\mathbf{E}\left[\sum\limits_{\{t\in J'_{l,u}\}}\mathbf{P}(\underline{t},\underline{S}^{(\underline{t})})\right]=\frac{l}{2}\sum\limits_{\{t\in J'_{l,u},t_{0}=1\}}\mathbf{E}\left[\mathbf{P}(\underline{t},\underline{S}^{(\underline{t})})\right]=\frac{l}{2}\left(\sum\limits_{t=1}^{u}\sum\limits_{x\in\mathbb{Z}}\mathbf{P}(S_{t}=x)^{2}\right)^{q}=\frac{l}{2}D(u)^{q}.
\end{split}
\end{equation}
Therefore, $\mathbf{E}[W_{l}]=1/2$. By Chebyshev's inequality, we have:
\begin{equation}\label{e234}
\mathbf{P}\left(W_{l}-\mathbf{E}[W_{l}]<\frac{1}{(1+\epsilon^{2})^{q}}-\mathbf{E}[W_{l}]\right)\leq 4\mathbf{V}\mbox{ar}(W_{l}),
\end{equation}

It remains to control the variance of $W_{l}$. We define
\begin{equation*}
Y_{j}=\frac{1}{D(u)^{q}}\sum\limits_{\underline{t}\in J'_{l,u}(j)}\mathbf{P}(\underline{t},\underline{S}^{(\underline{t})})-1,
\end{equation*}
where $J'_{l,u}(j)=\{\underline{t}\in J'_{l,u}: t_{0}=j\}$. It is obvious that $W_{l}-\mathbf{E}[W_{l}]=\left(\sum_{j=1}^{l/2}Y_{j}\right)/l$ and $\mathbf{E}[Y_{j}]=0$ by \eqref{e233}. Then we have
\begin{equation}\label{e235}
\mathbf{V}\mbox{ar}(W_{l})=\frac{1}{l^{2}}\sum\limits_{j_{1},j_{2}=1}^{\frac{l}{2}}\mathbf{E}[Y_{j_{1}}Y_{j_{2}}].
\end{equation}
By Gnedenko's local limit theorem (\textit{cf.} \cite[Theorem 8.4.1]{BGT89}), there exists a constant $C_{1}$, such that for any $t>0$ and $x\in\mathbb{Z}$,
\begin{equation}\label{e236}
\mathbf{P}(S_{t}=x)\leq\frac{C}{a_{t}}.
\end{equation}
Hence, by \eqref{e236}, \eqref{e112} and \eqref{e110},
\begin{equation}\label{e237}
Y_{j}\leq\frac{1}{D(u)^{q}}\sum\limits_{\underline{t}\in J'_{l,u}(j)}\mathbf{P}(\underline{t},\underline{S}^{(\underline{t})})\leq\frac{1}{D(u)^{q}}\left(\sum\limits_{t=1}^{u}\frac{C_{1}}{a_{t}}\right)^{q}\leq(C_{2})^{q}.
\end{equation}
Next, we will show that most summands in \eqref{e235} are zero. Note that for $j\in\{1,\cdots, l/2\}$, $t_{q}-t_{0}\leq qu$ for $t_{0},t_{q}\in J'_{l,u}(j)$. If we denote the increment of $S$ by $(Z_{n})_{n\geq1}$, then $Y_{j}$ only depends on $(Z_{j+1},\cdots,Z_{j+qu})$. Therefore, for $|j_{1}-j_{2}|>qu$, $Y_{j_{1}}$ and $Y_{j_{2}}$ are independent and $\mathbf{E}[Y_{j_{1}}Y_{j_{2}}]=\mathbf{E}[Y_{j_{1}}]\mathbf{E}[Y_{j_{2}}]=0$. By \eqref{e237},
\begin{equation*}
\mathbf{V}\mbox{ar}(W_{l})\leq\frac{qu}{l}(C_{2})^{2q}\leq q(C_{2})^{2q}l^{-\epsilon^{2}}.
\end{equation*}
Then \eqref{e234} is bounded above by $(C_{3})^{q}l^{-\epsilon^{2}}$, which tends to $0$ as $\beta$ tends to $0$ by the definition of $q$ and $l$ and we complete the proof of Lemma \ref{l204}.
\end{proof}
\subsection{Proof of Lemma \ref{l205}}In this subsection, we prove Lemma \ref{l205}. We will use $C$ to represent generic constants in the proof and it could change from line to line.
\begin{proof}[Proof of Lemma \ref{l205}]
For any trajectory of $S$, we shift the environment by
\begin{equation}\label{e238}
\hat{\omega}_{n,x}:=\omega_{n,x}-\lambda'(\beta)\mathbbm{1}_{\{S_{n}=x\}}.
\end{equation}
It is not hard to check that under $\mathbb{P}^{S}$, $\hat{\omega}$ is a family of independent random variables with mean $0$. Besides, when $\beta$ is small enough, by \eqref{e215} and \eqref{e216}, the variance of $\hat{\omega}_{n,x}$ can be bounded by $1+(\epsilon^{3}/2)$.

To bound $\mathbb{V}\mbox{ar}^{S}(X)$, we start by observing that
\begin{equation}\label{e239}
\mathbb{E}^{S}[X^{2}]=\frac{1}{2Rla_{l}D(u)^{q}}\mathbb{E}^{S}\left[\left(\sum\limits_{\underline{x}\in(\tilde{I}_{0})^{q+1},\underline{t}\in J_{l,u}}\mathbf{P}(\underline{t},\underline{x})\prod\limits_{i=0}^{q}\left(\hat{\omega}_{t_{j},x_{j}}+\lambda'(\beta)\mathbbm{1}_{\{S_{t_{j}}=x_{j}\}}\right)\right)^{2}\right].
\end{equation}
A simple expansion shows that
\begin{equation*}
\prod\limits_{i=0}^{q}\left(\hat{\omega}_{t_{j},x_{j}}+\lambda'(\beta)\mathbbm{1}_{\{S_{t_{j}}=x_{j}\}}\right)=\sum\limits_{r=0}^{q+1}(\lambda'(\beta))^{r}\sum\limits_{A\subset\{0,\cdots,q\}\atop|A|=r}\prod\limits_{k\in A}\mathbbm{1}_{\{S_{t_{k}}=x_{k}\}}\prod\limits_{j\in\{0,\cdots,q\}\backslash A}\hat{\omega}_{t_{j},x_{j}}.
\end{equation*}
Therefore, the square term in $\mathbb{E}^{S}$ in \eqref{e239} is the summation over $\underline{x},\underline{x}'\in(\tilde{I}_{0})^{q+1}, \underline{t},\underline{t}'\in J_{l,u}$ of $\mathbf{P}(\underline{t},\underline{x})\mathbf{P}(\underline{t}',\underline{x}')$ times
\begin{equation}\label{e240}
\sum\limits_{r=0}^{q+1}\sum\limits_{r'=0}^{q+1}(\lambda'(\beta))^{r+r'}\sum\limits_{A\subset\{0,\cdots,q\},|A|=r\atop B\subset\{0,\cdots,q\},|B|=r}\prod\limits_{k\in A\atop k'\in B}\mathbbm{1}_{\{S_{t_{k}}=x_{k}\}}\mathbbm{1}_{\{S_{t'_{k'}}=x'_{k'}\}}\prod\limits_{j\in\{0,\cdots,q\}\backslash A\atop j'\in\{0,\cdots,q\}\backslash B}\hat{\omega}_{t_{j},x_{j}}\hat{\omega}_{t'_{j'},x'_{j'}}.
\end{equation}
Note that $\hat{\omega}$ is a family of independent and mean-zero random variables under $\mathbb{P}^S$. When taking $\mathbb{P}^{S}$-expectation in \eqref{e240}, the summand is nonzero if and only if $r=r'$ and
\begin{equation*}
\{(t_{j},x_{j})|~j\in\{0,\cdots,q\}\backslash A\}=\{(t'_{j'},x'_{j'})|~j\in\{0,\cdots,q\}\backslash B\}.
\end{equation*}
Hence, to compute the $\mathbb{P}^{S}$-expectation of \eqref{e240}, we can first fix $(t_{j},x_{j})$ for $j\in\{0, \cdots, q\}\backslash A$, and then define a set of $(q-r+1)$-tuples:
\begin{equation*}
\mathcal{S}_{q-r}:=\{\underline{s}:=(s_{0},\cdots,s_{q-r}): 1\leq s_{0}<\cdots<s_{q-r}\leq l, s_{q-r}-s_{0}\leq qu\}.
\end{equation*}
For any given $\underline{s}\in\mathcal{S}_{q-r}$, we further define a related set of $r$-tuples:
\begin{equation*}
\mathcal{T}_{r}(\underline{s}):=\{\underline{t}=(t_{1},\cdots,t_{r}):~1\leq t_{1}<\cdots<t_{r}\leq l, \underline{s}\cdot\underline{t}\in J_{l,u}\},
\end{equation*}
where $\underline{s}\cdot\underline{t}$ is a $(q+1)$-tuple, which contains all the entries of $\underline{s}$ and $\underline{t}$ and the entries are ordered from the smallest to the largest.

Now we can have a nicer form for $\mathbb{V}$ar$^{S}(X)$. Note that the $\mathbb{P}^{S}$-expectation of the term $r=r'=q+1$ in \eqref{e240} is exactly the term $\mathbb{E}^{S}[X]^{2}$, so we can subtract it on both sides of \eqref{e239} and by recalling $\mathbb{E}^{S}[(\hat{\omega}_{n,x})^{2}]\leq(1+\epsilon^{3}/2)\leq2$ from \eqref{e238}, we obtain
\begin{equation}\label{e241}
\begin{split}
\mathbb{V}\mbox{ar}^{S}(X)\leq&\frac{(1+\epsilon^{3}/2)^{q+1}}{2Rla_{l}D(u)^{q}}\sum\limits_{\underline{x}\in(\tilde{I}_{0})^{q+1},\underline{t}\in J_{l,u}}\mathbf{P}(\underline{t},\underline{x})^{2}\\
+&\frac{1}{2Rla_{l}D(u)^{q}}\sum\limits_{r=1}^{q}(\lambda'(\beta))^{2r}2^{q+1-r}\\
&\sum\limits_{\underline{s}\in\mathcal{S}_{q-r}}\sum\limits_{\underline{x}\in(\tilde{I}_{0})^{q+1-r}}
\sum\limits_{\underline{t},\underline{t}'\in\mathcal{T}_{r}(\underline{s})}\mathbf{P}((\underline{s}\cdot\underline{t}),(\underline{x},\underline{S}^{(\underline{t})}))\mathbf{P}((\underline{s}\cdot\underline{t}'),(\underline{x},\underline{S}^{(\underline{t}')})),
\end{split}
\end{equation}
where the first term on the right-hand side of \eqref{e241} corresponds to $r=0$, and it is actually equal to $(1+\epsilon^{3}/2)^{q+1}\mathbb{E}[X^{2}]$ and bounded above by $(1+\epsilon^{3}/2)^{q+1}$. For the $(q+1)$-tuple $(\underline{x},\underline{S}^{(\underline{t})})$ in the last summation, its $i$-th element is $x_{j}$ if and only if the $i$-th element in $\underline{s}\cdot\underline{t}$ is $s_{j}$, while it is $S_{t_{j}}$ if and only if the $i$-th element in $\underline{s}\cdot\underline{t}$ is $t_{j}$.

Finally, we will bound
\begin{equation}\label{e242}
\begin{split}
&\sum\limits_{\underline{s}\in\mathcal{S}_{q-r}}\sum\limits_{\underline{x}\in(\tilde{I}_{0})^{q+1-r}}
\sum\limits_{\underline{t},\underline{t}'\in\mathcal{T}_{r}(\underline{s})}\mathbf{P}((\underline{s}\cdot\underline{t}),(\underline{x},\underline{S}^{(\underline{t})}))\mathbf{P}((\underline{s}\cdot\underline{t}'),(\underline{x},\underline{S}^{(\underline{t}')}))\\
=&\sum\limits_{\underline{s}\in\mathcal{S}_{q-r}}\sum\limits_{\underline{x}\in(\tilde{I}_{0})^{q+1-r}}\left(\sum\limits_{\underline{t}\in\mathcal{T}_{r}(\underline{s})}\mathbf{P}((\underline{s}\cdot\underline{t}),(\underline{x},\underline{S}^{(\underline{t})}))\right)^{2},
\end{split}
\end{equation}
which is the most complicated part of the proof.

First, let us denote $s_{-1}:=0$ and $s_{q-r+1}:=l$. We can split the summation $\sum_{\underline{t}\in\mathcal{T}_{r}(\underline{s})}\mathbf{P}((\underline{s}\cdot\underline{t}),(\underline{x},\underline{S}^{(\underline{t})}))$ according to the position of $t_{1}$. We have
\begin{equation*}
\sum\limits_{\underline{t}\in\mathcal{T}_{r}(\underline{s})}\mathbf{P}((\underline{s}\cdot\underline{t}),(\underline{x},\underline{S}^{(\underline{t})}))=\sum\limits_{k=0}^{q+1-r}\sum\limits_{\underline{t}\in\mathcal{T}_{r}(\underline{s}),t_{1}\in(s_{k-1},s_{k})}\mathbf{P}((\underline{s}\cdot\underline{t}),(\underline{x},\underline{S}^{(\underline{t})})).
\end{equation*}
We observe that
\begin{equation}\label{e243}
\begin{split}
&\sum\limits_{\underline{t}\in\mathcal{T}_{r}(\underline{s}),t_{1}\in(s_{k-1},s_{k})}\mathbf{P}((\underline{s}\cdot\underline{t}),(\underline{x},\underline{S}^{(\underline{t})}))\leq\sum\limits_{0= m_{0}=\cdots=m_{k-1}<\atop m_{k}\leq m_{k+1}\leq\cdots\leq m_{q-r}\leq r}\mathbbm{1}_{\underline{t}\in\mathcal{T}_{r}(\underline{s})}\\
\prod\limits_{i=1}^{q-r}&\sum\limits_{s_{i-1}<t_{m_{i-1}+1}<\cdots<t_{m_{i}}<s_{i}}\mathbf{P}((s_{i-1},t_{m_{i-1}+1},\cdots,t_{m_{i}},s_{i}),(x_{i-1},S_{t_{m_{i-1}+1}},\cdots,S_{t_{m_{i}}},x_{i}))\\
&\times\sum\limits_{0<t_{1}<\cdots<t_{m_{0}}<s_{0}}\mathbf{P}((t_{1},\cdots,t_{m_{0}},s_{0}),(S_{t_{1}},\cdots,S_{t_{m_{0}}},x_{0}))\\
&\times\sum\limits_{s_{q-r}<t_{m_{q-r}+1}<\cdots<t_{r}<l}\mathbf{P}((s_{q-r},t_{m_{q-r}+1},\cdots,t_{r}),(x_{q-r},S_{t_{m_{q-r}+1}},\cdots,S_{t_{r}})).
\end{split}
\end{equation}
Here $m_{i}$ denotes the number of $t$-indices before $s_{i}$. If $m_{0}=0$, then the third line of \eqref{e243} is simply $1$ and so is the fourth line of \eqref{e243} if $m_{q-r}=r$. Note that 

We can bound the factor in the second line of \eqref{e243} for any $i\in\{1,\cdots,q-r\}$ according to the following lemma.
\begin{lemma}\label{l206}
There exists a constant $C$, such that for any $j\in\mathbb{N}$ and any $(z_{i})_{i=1}^{j}\in\mathbb{Z}^{j}$,
\begin{equation}\label{e244}
\sum\limits_{0<t_{1}<\cdots<t_{j}<s\atop|t_{i}-t_{i-1}|\leq u, i=1,\cdots,j}\mathbf{P}((0,t_{1},\cdots,t_{j},s),(0,z_{1},\cdots,z_{j},x))\leq(CD(u))^{j}p_{s}(0,x),
\end{equation}
where $t_{0}:=0$ for convention and we use the notation
\begin{equation*}
p_{t}(x,y)=\mathbf{P}(S_{t}=y-x)
\end{equation*}
for any $t\geq1$ and $y,x\in\mathbb{Z}$.
\end{lemma}
\begin{proof}[Proof of Lemma \ref{l206}]
Recall the definition \eqref{e209} for $\mathbf{P}(\underline{t},\underline{x})$ and note that the product of the first two factors of $\mathbf{P}((0,t_{1},\cdots,t_{j},s),(0,z_{1},\cdots,z_{j},x))$ is
\begin{equation}\label{e245}
\mathbf{P}(S_{t_{1}}=z_{1})\mathbf{P}(S_{t_{2}}-S_{t_{1}}=z_{2}-z_{1}).
\end{equation}
We now show an upper bound for \eqref{e245} when it is non-zero. By Gnedenko's local limit theorem (\textit{cf.} \cite[Theorem 8.4.1]{BGT89}), there exists a constant $C$, such that for all $t\in\mathbb{N}$ and any $|x|\leq 2a_{t}$ with $\mathbf{P}(S_{t}=x)\neq0$,
\begin{equation}\label{e246}
\mathbf{P}(S_{t}=x)\geq\frac{C}{a_{t}}.
\end{equation}
When $|z_{2}|\leq 2a_{t_{2}}$, by \eqref{e236} and \eqref{e246}, we have
\begin{equation}\label{e247}
\frac{\mathbf{P}(S_{t_{1}}=z_{1})\mathbf{P}(S_{t_{2}}-S_{t_{1}}=z_{2}-z_{1})}{\mathbf{P}(S_{t_{2}}=z_{2})}\leq C\frac{a_{t_{2}}}{a_{t_{1}}a_{t_{2}-t_{1}}}=C\frac{t_{2}\varphi(t_{2})}{t_{1}\varphi(t_{1})(t_{2}-t_{1})\varphi(t_{2}-t_{1})}.
\end{equation}
Suppose $t_{1}\geq t_{2}-t_{1}$. Then $t_{2}/t_{1}\leq2$. By Potter bounds (\textit{cf}. \cite[Theorem 1.5.6]{BGT89}),
\begin{equation}\label{e248}
\frac{\mathbf{P}(S_{t_{1}}=z_{1})\mathbf{P}(S_{t_{2}}-S_{t_{1}}=z_{2}-z_{1})}{\mathbf{P}(S_{t_{2}}=z_{2})}\leq\frac{C}{a_{t_{1}}\wedge a_{t_{2}-t_{1}}}.
\end{equation}
When $|z_{2}|\geq 2a_{t_{2}}$, by \eqref{e115},
\begin{equation}\label{e249}
\mathbf{P}(S_{t_{2}}=z_{2})\geq Ct_{2}L(|z_{2}|)/(z_{2})^{2}.
\end{equation}
Suppose $|z_{1}|\geq|z_{2}-z_{1}|$. Then $|z_{1}|\geq a_{t_{2}}\geq a_{t_{1}}$. We can apply the upper bound in \eqref{e115} to $\mathbf{P}(S_{t_{1}}=z_{1})$ and apply \eqref{e236} to $\mathbf{P}(S_{t_{2}-t_{1}}=z_{2}-z_{1})$, and then by \eqref{e249}, we have
\begin{equation*}
\frac{\mathbf{P}(S_{t_{1}}=z_{1})\mathbf{P}(S_{t_{2}}-S_{t_{1}}=z_{2}-z_{1})}{\mathbf{P}(S_{t_{2}}=z_{2})}\leq\frac{t_{1}(z_{2})^{2}L(|z_{1}|)}{t_{2}(z_{1})^{2}L(|z_{2}|)}\frac{C}{a_{t_{2}-t_{1}}}.
\end{equation*}
Since $t_{1}/t_{2}\leq 1$ and $|z_{2}|/|z_{1}|\leq 2$, by Potter bounds (\textit{cf}. \cite[Theorem 1.5.6]{BGT89}), \eqref{e248} also holds, i.e., we have establish \eqref{e248} for any $z_{2}\in\mathbb{Z}$.

Then, by \eqref{e248}, \eqref{e112} and \eqref{e110}, we have
\begin{equation*}
\begin{split}
&\sum\limits_{0<t_{1}<\cdots<t_{j}<s\atop|t_{i}-t_{i-1}|\leq u, i=1,\cdots,j}\mathbf{P}((0,t_{1},\cdots,t_{j},s),(0,z_{1},\cdots,z_{j},x))\\
\leq&\sum\limits_{0<t_{1}<\cdots<t_{j}<s\atop|t_{i}-t_{i-1}|\leq u, i=1,\cdots,j}\frac{C}{a_{t_{1}}\wedge a_{t_{2}-t_{1}}}\mathbf{P}((0,t_{2},\cdots,t_{j},s),(0,z_{2},\cdots,z_{j},x))\\
\leq&\sum\limits_{0<t_{1}<2u}\frac{C}{a_{t_{1}}\wedge a_{2u-t_{1}}}\sum\limits_{0<t_{2}<\cdots<t_{j}<s\atop|t_{i}-t_{i-1}|\leq u, i=2,\cdots,j, t_{1}:=0}\mathbf{P}((0,t_{2},\cdots,t_{j},s),(0,z_{2},\cdots,z_{j},x))\\
\leq&2CD(u)\sum\limits_{0<t_{2}<\cdots<t_{j}<s\atop|t_{i}-t_{i-1}|\leq u, i=2,\cdots,j, t_{1}:=0}\mathbf{P}((0,t_{2},\cdots,t_{j},s),(0,z_{2},\cdots,z_{j},x)).
\end{split}
\end{equation*}
By induction, we then prove \eqref{e244}.
\end{proof}
The case $r=q$ in \eqref{e241} will be dealt with later. For $1\leq r\leq q-1$ in \eqref{e241}, i.e. $|\underline{s}|\geq2$, we apply Lemma \ref{l206} for all terms in \eqref{e243} with $t,s$-indices larger than $s_{k}$ to obtain an upper bound
\begin{equation}\label{e250}
\begin{split}
\mathbbm{1}_{\underline{t}\in\mathcal{T}_{r}(\underline{s})}&\sum\limits_{0= m_{0}=\cdots=m_{k-1}<\atop m_{k}\leq m_{k+1}\leq\cdots\leq m_{q-r}\leq r}\left(\sum\limits_{0<t_{1}<\cdots<t_{m_{0}}<s_{0}}\mathbf{P}((t_{1},\cdots,t_{m_{0}},s_{0}),(S_{t_{1}},\cdots,S_{t_{m_{0}}},x_{0}))\right)\\
\times&\prod\limits_{i=1}^{k-1}p_{s_{i}-s_{i-1}}(x_{i-1},x_{i})\\
\times&\sum\limits_{s_{k-1}<t_{m_{k-1}+1}<\cdots<t_{m_{k}}<s_{k}}\mathbf{P}((x_{k-1},S_{t_{m_{k-1}+1}},\cdots,S_{m_{k}},x_{k}))\\
\times&(CD(u))^{m_{q-r}-m_{k}}\prod\limits_{i=k+1}^{q-r}p_{s_{i}-s_{i-1}}(x_{i-1},x_{i})\\
\times&\sum\limits_{s_{q-r}<t_{m_{q-r}+1}<\cdots<t_{r}<l}\mathbf{P}((s_{q-r},t_{m_{q-r}+1},\cdots,t_{r}),(x_{q-r},S_{t_{m_{q-r}+1}},\cdots,S_{t_{r}})).
\end{split}
\end{equation}
Recall that the factor in the first line of \eqref{e250} is $1$ if $m_{0}=0$ and note that if $m_{q-r}<r$, i.e. $t_{1}<s_{q-r}$, we should further bound the last line in \eqref{e250} from above by
$(CD(u))^{r-m_{q-r}}$, which is due to \eqref{e112} and \eqref{e110}.

Note that the number of possible interlacements of $0\leq m_{0}\leq\cdots\leq m_{q-r}\leq r$ is not larger than $2^{q}$. Hence, according to the value of $k$, \eqref{e250} can be bounded above by
\begin{equation}\label{e251}
\begin{split}
J_{0}=2^{q}\mathbbm{1}_{\underline{t}\in\mathcal{T}_{r}(\underline{s})}&\sum\limits_{m_{k}=1}^{r}(CD(u))^{r-m_{k}}\times\\
&\sum\limits_{0<t_{1}<\cdots<t_{m_{0}}<s_{0}}\mathbf{P}((t_{1},\cdots,t_{m_{0}},s_{0}),(S_{t_{1}},\cdots,S_{t_{m_{0}}},x_{0}))\prod\limits_{i=1}^{q-r}p_{s_{i}-s_{i-1}}(x_{i-1},x_{i})
\end{split}
\end{equation}
if $k=0$;
\begin{equation}\label{e252}
\begin{split}
J_{k}=2^{q}\mathbbm{1}_{\underline{t}\in\mathcal{T}_{r}(\underline{s})}\sum\limits_{m_{k}=1}^{r}&(CD(u)^{r-m_{k}})\prod\limits_{i=1}^{k-1}p_{s_{i}-s_{i-1}}(x_{i-1}-x_{i})\\
&\times\sum\limits_{s_{k-1}<t_{1}<\cdots<t_{m_{k}}<s_{k}}\mathbf{P}((s_{k-1},t_{1},\cdots,t_{m_{k}},s_{k}))\\
&\times\prod\limits_{i=k+1}^{q-r}p_{s_{i}-s_{i-1}}(x_{i-1},x_{i})\\
\end{split}
\end{equation}
if $1\leq k\leq q-r$; and
\begin{equation}\label{e253}
\begin{split}
J_{q+1-r}=2^{q}\mathbbm{1}_{\underline{t}\in\mathcal{T}_{r}(\underline{s})}\sum\limits_{m_{k}=1}^{r}&\prod\limits_{i=1}^{q-r}p_{s_{i}-s_{i-1}}(x_{i-1},x_{i})\\
&\times\sum\limits_{s_{q-r}<t_{1}<\cdots<t_{r}\leq l}\mathbf{P}((s_{q-r},t_{1},\cdots,t_{r}),(x_{q-r},S_{t_{1}},\cdots,S_{t_{r}}))
\end{split}
\end{equation}
if $k=q+1-r$.

Now we can expand the square term in \eqref{e242} and then bound \eqref{e242} from above by
\begin{equation*}
\sum\limits_{\underline{s}\in\mathcal{S}_{q-r}}\sum\limits_{\underline{x}\in(\tilde{I}_{0})^{q+1-r}}\sum\limits_{k,k'=1}^{q+1-r}\sum\limits_{m_{k},m'_{k'}=1}^{r}J_{k}J_{k'},
\end{equation*}
where the expressions for $J_{k}$ and $J_{k'}$ can be \eqref{e251},\eqref{e252}, or \eqref{e253}. We will use different summing strategies to bound
\begin{equation}\label{e2a}
\sum\limits_{\underline{s}\in\mathcal{S}_{q-r}}\sum\limits_{\underline{x}\in(\tilde{I}_{0})^{q+1-r}}\sum\limits_{m_{k},m'_{k'}=1}^{r}J_{k}J_{k'}
\end{equation}
for different $k$ and $k'$. There are two basic cases:
\begin{equation*}
\begin{split}
&\mbox{{\bf Case A:}}~k=k',\\
&\mbox{{\bf Case B:}}~k\neq k',
\end{split}
\end{equation*}
and we start by bounding {\bf Case A:}~$k=k'$.

According to the value of $k$ and $k'$, there are three sub-cases of {\bf Case A}:
\begin{equation*}
\begin{split}
&\mbox{{\bf Case A1:}}~k=k=0,\\
&\mbox{{\bf Case A2:}}~k=k'=q+1-r,\\
&\mbox{{\bf Case A3:}}~1\leq k=k'\leq q-r.
\end{split}
\end{equation*}

{\bf Case A1:}~$k=k'=0$:

If $k=k'=0$ in \eqref{e2a}, then we can first fix the position of $s_{0}$, which has at most $l$ choices. Note that we have the term $\prod_{i=1}^{q-r}\sum_{(x_{1},\cdots,x_{q-r})\in(\tilde{I}_{0})^{q-r}}(p_{s_{i}-s_{i-1}}(x_{i-1},x_{i}))^{2}$. Hence, for any $x_{0}$, we can sum over $s_{1},\cdots,s_{q-r}$ and $x_{1},\cdots,x_{q-r}$ by \eqref{e110}, which gives $\prod_{i=1}^{q-r}D(s_{i}-s_{i-1})$. By Potter bounds \cite[Theorem 1.5.6]{BGT89},
\begin{equation}\label{e254}
\prod\limits_{i=1}^{q-r}D(s_{i}-s_{i-1})\leq (CD(u))^{q-r}\prod\limits_{i=1}^{q-r}\frac{s_{i}-s_{i-1}}{u}.
\end{equation}
Since $s_{q-r}\leq qu$, $\prod_{i=1}^{q-r}((s_{i}-s_{i-1})/u)\leq(q/(q-r))^{q-r}$. Hence, \eqref{e254} is bounded above by $C^{q}D(u)^{q-r}$.

Next, we use the trivial bound
\begin{equation*}
\sum\limits_{0<t'_{1}<\cdots<t'_{m'_{0}}<s_{0}}\mathbf{P}((t'_{1},\cdots,t'_{r},s_{0}),(S_{t'_{1}},\cdots,S_{t'_{m'_{0}}},x_{0}))\leq(CD(u))^{m'_{0}}
\end{equation*}
and then sum over $s_{0}-t_{m_{0}}$ and $x_{0}$ by
\begin{equation*}
\sum\limits_{s_{0}-t_{m_{0}}=1}^{u}\sum\limits_{x_{0}\in\tilde{I}_{0}}p_{s_{0}-t_{m_{0}}}(S_{t_{r}},x_{0})\leq\sum\limits_{t=1}^{u}1=u.
\end{equation*}
At last, we use the trivial bound
\begin{equation*}
\sum\limits_{0<t_{1}<\cdots<t_{m_{0}}}\mathbf{P}((t_{1},\cdots,t_{m_{0}}),(S_{t_{1}},\cdots,S_{t_{m_{0}}}))\leq(CD(u))^{m_{0}-1}.
\end{equation*}
Now we obtain that for any $m_{0}$ and $m'_{0}$,
\begin{equation*}
\sum\limits_{\underline{s}\in\mathcal{S}_{q-r}}\sum\limits_{x\in(\tilde{I}_{0})^{q+1-r}}(J_{0})^{2}\leq C^{q}ulD(u)^{q+r-1}.
\end{equation*}

{\bf Case A2:}~$k=k'=q+1-r$:

If $k=k'=q+1-r$ in \eqref{e2a}, then we can first fix the position of $s_{q-r}$ and then apply the strategy above to obtain that
\begin{equation*}
\sum\limits_{\underline{s}\in\mathcal{S}_{q-r}}\sum\limits_{x\in(\tilde{I}_{0})^{q+1-r}}(J_{q+1-r})^{2}\leq C^{q}ulD(u)^{q+r-1}.
\end{equation*}

{\bf Case A3:}~$1\leq k=k'\leq q-r$:

If $1\leq k=k'\leq q-r$ in \eqref{e2a}, then $s_{k-1}<t_{1}$ and $s_{k-1}<t'_{1}$ by \eqref{e252} and we can first fix the position of $s_{k-1}$, which has at most $l$ choices. Note that we have the term $\prod_{i=k-1}^{1}\sum_{(x_{k-2},\cdots,x_{0})\in(\tilde{I}_{0})^{k-1}}(p_{s_{i}-s_{i-1}}(x_{i-1},x_{i}))^{2}$. Hence, for any $x_{k-1}$, we can sum over $s_{0},\cdots,s_{k-2}$ and $x_{0},\cdots,x_{k-2}$ by \eqref{e110} and \eqref{e254} (hold $(s_{k-1},x_{k-1})$ for the moment), which gives $C^{q}D(u)^{k-1}$. For the same reason, then we can sum over $s_{k+1},\cdots,s_{q-r}$ and $x_{k+1},\cdots,x_{q-r}$ (hold $(s_{k},x_{k})$ for the moment), which gives $C^{q}D(u)^{q-r-k}$ . These summations and products together give $C^{q}D(u)^{q-r-1}$.

Next, we apply Lemma \ref{l206} to obtain
\begin{equation*}
\begin{split}
\sum\limits_{s_{k-1}<t'_{1}<\cdots<t'_{m'_{k}}<s_{k}}&\mathbf{P}((s_{k-1},t'_{1},\cdots,t'_{m'_{k}},s_{k}),(x_{k-1},S_{t'_{1}},\cdots,S_{t'_{m'_{k}}},x_{k}))\\
&\leq(CD(qu))^{m'_{k}}p_{s_{k}-s_{k-1}}(x_{i-1},x_{i}).
\end{split}
\end{equation*}
Then it remains to bound
\begin{equation}\label{e255}
\begin{split}
\sum\limits_{s_{k}-s_{k-1}=1}^{m_{k}u}&\sum\limits_{s_{k-1}<t_{1}<\cdots<t_{m_{k}}<s_{k}}\sum\limits_{x_{k-1},x_{k}\in\tilde{I}_{0}}\\
&p_{s_{k}-s_{k-1}}(x_{k-1},x_{k})\mathbf{P}((s_{k-1},t_{1},\cdots,t_{m_{k}},s_{k}),(x_{k-1},S_{t_{1}},\cdots,S_{t_{m_{k}}},x_{k})).
\end{split}
\end{equation}
Note that by \eqref{e114}, there exists a $T>0$, such that for all $t\geq T$, $P(S_{t}=x)>0$ for any $x\in\mathbb{Z}$. Hence, we can split \eqref{e255} into three parts:

(i) $\min\{t_{1}-s_{k-1}, s_{k}-t_{m_{k}}\}\geq T$.

(ii) $\min\{t_{1}-s_{k-1}, s_{k}-t_{m_{k}}\}<T$ and $\max\{t_{1}-s_{k-1}, s_{k}-t_{m_{k}}\}\geq T$.

(iii) $\max\{t_{1}-s_{k-1}, s_{k}-t_{m_{k}}\}<T$.

To deal with part (i) in \eqref{e255}, we need the following lemma.
\begin{lemma}\label{l207}
For any $\epsilon>0$, there exists a constant $C$, such that for any $k\geq2$ and all $n\geq k$,
\begin{equation}\label{e256}
\sum\limits_{j_{1}+\cdots+j_{k}=n\atop j_{i}>0, \forall i\in\{1,\cdots,k\}}\frac{1}{a_{j_{1}+j_{2}+n}}\left(\prod\limits_{i=3}^{k}\frac{1}{a_{j_{i}}}\mathbbm{1}_{\{k\geq3\}}+\mathbbm{1}_{\{k<3\}}\right)\leq n^{\epsilon^{4}}C^{k-1}D(n)^{k-2}.
\end{equation}
\end{lemma}
\begin{proof}[Proof of Lemma \ref{l207}]
We prove it by induction.

For $k=2$, by Potter bounds \cite[Theorem 1.5.6]{BGT89},
\begin{equation*}
\sum\limits_{j_{1}+j_{2}=n\atop j_{1},j_{2}>0}\frac{1}{a_{j_{1}+j_{2}+n}}=\frac{n-1}{a_{2n}}\leq\frac{1}{2\phi(2n)}\leq Cn^{\epsilon^{4}}.
\end{equation*}
Suppose \eqref{e256} is valid for $k\geq2$ and then for $k+1$, since $a_{(\cdot)}$ is increasing,
\begin{equation*}
\begin{split}
&\sum\limits_{j_{1}+\cdots+j_{k+1}=n\atop j_{i}>0, \forall i\in\{1,\cdots,k+1\}}\frac{1}{a_{j_{1}+j_{2}+n}}\prod\limits_{i=3}^{k+1}\frac{1}{a_{j_{i}}}\\
\leq&\sum\limits_{j_{k+1}=1}^{n-k}\frac{1}{a_{j_{k+1}}}\sum\limits_{j_{1}+\cdots+j_{k}=n-j_{k+1}\atop j_{i}>0, \forall i\in\{1,\cdots,k\}}\frac{1}{a_{j_{1}+j_{2}+n-j_{k+1}}}\prod\limits_{i=3}^{k}\frac{1}{a_{j_{i}}}\\
\leq&\sum\limits_{j_{k+1}=1}^{n-k}\frac{1}{a_{j_{k+1}}}(n-j_{k+1})^{\epsilon^{4}}C^{k-1}D(n-j_{k+1})^{k-2}\\
\leq&n^{\epsilon^{4}}C^{k}D(n))^{k-1}.
\end{split}
\end{equation*}
Then the induction is completed and the lemma has been proved.
\end{proof}
Since $\min\{t_{1}-s_{k-1}, s_{k}-t_{m_{k}}\}\geq T$, we have
\begin{equation*}
\begin{split}
&\sum\limits_{x_{k-1},x_{k}\in\tilde{I}_{0}}p_{t_{1}-s_{k-1}}(x_{k-1},S_{t_{1}})p_{s_{k}-s_{k-1}}(x_{k-1},x_{k})p_{s_{k}-t_{m_{k}}}(S_{t_{m_{k}}},x_{k})\\
\leq&C p_{t_{1}-s_{k-1}+s_{k}-s_{k-1}+s_{k}-t_{m_{k}}}(S_{t_{1}},S_{t_{m_{k}}})\leq\frac{C}{a_{t_{1}-s_{k-1}+s_{k}-t_{m_{k}}+s_{k}-s_{k-1}}},
\end{split}
\end{equation*}
where
\begin{equation*}
\begin{split}
&p_{t_{m_{k=1}+1}-s_{k-1}}(x_{k-1},S_{t_{1}})\leq Cp_{t_{1}-s_{k-1}}(S_{t_{1}},x_{k-1})\\ &p_{s_{k}-t_{m_{k}}}(S_{t_{m_{k}}},x_{k})\leq Cp_{s_{k}-t_{m_{k}}}(x_{k},S_{t_{m_{k}}})\\
\end{split}
\end{equation*}
follow from the arguments \eqref{e246}-\eqref{e249}. Then, by Lemma \ref{l207}, part (i) in \eqref{e255} is bounded above by
\begin{equation}\label{e257}
\begin{split}
&C\sum\limits_{s_{k}-s_{k-1}=1}^{m_{k}u}\sum\limits_{s_{k-1}<t_{1}<\cdots<t_{m_{k}}<s_{k}}\frac{1}{a_{t_{1}-s_{k-1}+s_{k}-t_{m_{k}}+s_{k}-s_{k-1}}}\prod\limits_{i=2}^{t_{m_{k}}}\frac{1}{a_{t_{i}-t_{i-1}}}\\
\leq&C^{m_{k}}(m_{k}u)^{1+\epsilon^{4}}(D(m_{k}u))^{m_{k}-1}\leq C^{m_{k}}(qu)^{1+\epsilon^{4}}(D(qu))^{m_{k}-1}.
\end{split}
\end{equation}

We will use the following lemma to handle $D(qu)$.
\begin{lemma}\label{l208}
Recall that $u\to\infty$ as the inverse temperature $\beta\to0$. We have
\begin{equation}\label{e258}
\lim\limits_{\beta\to 0}\frac{D(qu)}{D(u)}=1
\end{equation}
\end{lemma}
\begin{proof}[Proof of Lemma \ref{l208}]
Without loss of generality, we may assume that $D(\cdot)$ and $\varphi(\cdot)$ are differentiable by \cite[Theorem 1.8.2]{BGT89}. Then by definition of $D(\cdot)$, it follows that $D'(u)\sim(u\varphi(u))^{-1}$.

We will apply \cite[Proposition 2.3.2, Theorem 2.3.1]{BGT89} to prove \eqref{e258}, which reduces \eqref{e258} to showing
\begin{equation*}
\lim\limits_{\beta\to 0}\frac{uD'(u)\log q}{D(u)}=0.
\end{equation*}
By recalling the definition of $u$ and $q$ from \eqref{e206}, we need to show
\begin{equation}\label{e259}
\lim\limits_{\beta\to 0}\max\left\{\frac{f_{1}(u):=\log\left(\log\sqrt{\varphi\left(u^{\frac{1}{1-\epsilon^{2}}}\right)}\vee e\right)}{\varphi(u)D(u)},\frac{f_{2}(u):=\log\log{D\left(u^{\frac{1}{1-\epsilon^{2}}}\right)}}{\varphi(u)D(u)}\right\}=0.
\end{equation}
We will prove \eqref{e259} by proving both $f_{1}(u)/\varphi(u)D(u)$ and $f_{2}(u)/\varphi(u)D(u)$ tend to $0$ as $\beta$ tends to $0$.

For $f_{1}(u)/\varphi(u)D(u)$, note that $\varphi(u)D(u)\to\infty$ as $\beta\to0$. Then by L'Hospital rule, we have
\begin{equation*}
\begin{split}
\lim\limits_{\beta\to 0}\frac{\log\log\sqrt{\varphi\left(u^{\frac{1}{1-\epsilon^{2}}}\right)}}{\varphi(u)D(u)}&=\lim\limits_{\beta\to 0}\frac{1}{\log\varphi\left(u^{\frac{1}{1-\epsilon^{2}}}\right)}\frac{1}{\varphi\left(u^{\frac{1}{1-\epsilon^{2}}}\right)}\frac{u^{\frac{\epsilon^{2}}{1-\epsilon^{2}}}\varphi\left(u^{\frac{1}{1-\epsilon^{2}}}\right)}{(1-\epsilon^{2})(\varphi'(u)D(u)+\varphi(u)D'(u))}\\
&=\lim\limits_{\beta\to 0}\frac{1}{\log\varphi\left(u^{\frac{1}{1-\epsilon^{2}}}\right)}\frac{1}{\varphi\left(u^{\frac{1}{1-\epsilon^{2}}}\right)}\frac{u^{\frac{1}{1-\epsilon^{2}}}\varphi\left(u^{\frac{1}{1-\epsilon^{2}}}\right)}{(1-\epsilon^{2})u(\varphi'(u)D(u)+1)}=0,
\end{split}
\end{equation*}
where we use the property that $\lim_{x\to\infty}x\varphi'(x)/\varphi(x)=0$ by \cite[Section 1.8]{BGT89}.

By the same computation as above, we also have
\begin{equation*}
\lim\limits_{\beta\to 0}\frac{\log\log D(u^{\frac{1}{1-\epsilon^{2}}})}{\varphi(u)D(u)}=0
\end{equation*}
and thus \eqref{e258} is proved.
\end{proof}
By Lemma \ref{l208}, \eqref{e257} can be bounded above by $(2C)^{m_{k}}(qu)^{1+\epsilon^{4}}D(u)^{m_{k}-1}$.

For part (ii) in \eqref{e255}, let us assume $s_{k}-t_{m_{k}}\geq T$ and $t_{1}-s_{k-1}<T$. Then
\begin{equation*}
\begin{split}
&\sum\limits_{x_{k-1},x_{k}\in\tilde{I}_{0}}p_{t_{1}-s_{k-1}}(x_{k-1},S_{t_{1}})p_{s_{k}-s_{k-1}}(x_{k-1},x_{k})p_{s_{k}-t_{m_{k}}}(S_{t_{m_{k}}},x_{k})\\
\leq&C\sum\limits_{x_{k-1}\in\tilde{I}_{0}}p_{t_{1}-s_{k-1}}(x_{k-1},S_{t_{1}}) p_{s_{k}-s_{k-1}+s_{k}-t_{m_{k}}}(x_{k-1},S_{t_{m_{k}}})\\
\leq&\sum\limits_{x_{k-1}\in\tilde{I}_{0}}p_{t_{1}-s_{k-1}}(x_{k-1},S_{t_{1}})\frac{C}{a_{s_{k}-t_{m_{k}}+s_{k}-s_{k-1}}}\leq\frac{C}{a_{s_{k}-t_{m_{k}}+s_{k}-s_{k-1}}}.
\end{split}
\end{equation*}
It is not hard to check that by the proof of Lemma \ref{l207}, it follows that
\begin{equation*}
\sum\limits_{j_{1}+\cdots+j_{k}=n\atop j_{i}>0, \forall i\in\{1,\cdots,k\}}\frac{1}{a_{j_{1}+n}}\left(\prod\limits_{i=2}^{k}\frac{1}{a_{j_{i}}}\mathbbm{1}_{\{k\geq2\}}+\mathbbm{1}_{\{k<2\}}\right)\leq n^{\epsilon^{4}}C^{k-1}D(n)^{k-1}.
\end{equation*}
Hence, part (ii) in \eqref{e255} can be bounded above by $TC^{m_{k}-1}(qu)^{1+\epsilon^{4}}D(u)^{m_{k}-1}$, where $T$ comes from $\sum_{t_{1}-s_{t-1}=1}^{T}$.

For part (iii) in \eqref{e255}, we have
\begin{equation*}
\begin{split}
&\sum\limits_{x_{k-1},x_{k}\in\tilde{I}_{0}}p_{t_{1}-s_{k-1}}(x_{k-1},S_{t_{1}})p_{s_{k}-s_{k-1}}(x_{k-1},x_{k})p_{s_{k}-t_{m_{k}}}(S_{t_{m_{k}}},x_{k})\\
\leq&\frac{C}{a_{s_{k}-s_{k-1}}}\sum\limits_{x_{k-1},x_{k}\in\tilde{I}_{0}}p_{t_{1}-s_{k-1}}(x_{k-1},S_{t_{1}})p_{s_{k}-t_{m_{k}}}(S_{t_{m_{k}}},x_{k})\leq\frac{C}{a_{s_{k}-s_{k-1}}}.
\end{split}
\end{equation*}
Similarly, part (iii) can be bounded above by $T^{2}C^{m_{k}-2}(qu)^{1+\epsilon^{4}}D(u)^{m_{k}-1}$. Hence, \eqref{e255} can be bounded above by $C^{m_{k}}(qu)^{1+\epsilon^{4}}D(u)^{m_{k}-1}$ and we obtain that for any $m_{k}$ and $m'_{k}$,
\begin{equation*}
\sum\limits_{\underline{s}\in\mathcal{S}_{q-r}}\sum\limits_{\underline{x}\in(\tilde{I}_{0})^{q+1-r}}(J_{k})^{2}\leq C^{q}(qu)^{1+\epsilon^{4}}lD(u)^{q+r-1},
\end{equation*}
which finished the estimate for {\bf Case A3}.

Now all sub-cases of {\bf Case A} have been handled and we start to consider {\bf Case B} for \eqref{e2a}. Recall that $k\neq k'$ in {\bf Case B} and we may just assume that $k<k'$. First, we can fix the position of $s_{k-1}$, which has at most $l$ choices. Next, if $k'=q+1-r$, then we just use the trivial bound
\begin{equation*}
\sum\limits_{s_{q-r}<t'_{1}<\cdots<t'_{r}\leq l}\mathbf{P}((s_{q-r},t'_{1},\cdots,t'_{r}),(x_{q-r},S_{t'_{1}},\cdots,S_{t'_{r}}))\leq(CD(u))^{r},
\end{equation*}
while if $k'<q+1-r$, then we apply Lemma \ref{l206} to obtain
\begin{equation*}
\begin{split}
\sum\limits_{s_{k'-1}<t'_{1}<\cdots<t'_{m'_{k'}}<s_{k'}}&\mathbf{P}((s_{k'-1},t'_{1},\cdots,t'_{m'_{k'}},s_{k'}),(x_{k'-1},S_{t'_{1}},\cdots,S_{t'_{m'_{k'}}},x_{k'}))\\
&\leq(CD(u))^{m'_{k'}}p_{s_{k'}-s_{k'-1}}(x_{k'-1},x_{k'}).
\end{split}
\end{equation*}

According to the value of $k$, there are two sub-cases in {\bf Case B}:
\begin{equation*}
\begin{split}
&\mbox{{\bf Case B1:}}~k=0,\\
&\mbox{{\bf Case B2:}}~k>0.
\end{split}
\end{equation*}

{\bf Case B1}:

If $k=0$ in \eqref{e2a}, then we have the term $\prod_{i=1}^{q-r}\sum_{(x_{1},\cdots,x_{q-r})\in(\tilde{I}_{0})^{q-r}}(p_{s_{i}-s_{i-1}}(x_{i-1},x_{i}))^{2}$ and for any $x_{0}$, we can sum over $s_{1},\cdots,s_{q-r}$ and $x_{1},\cdots,x_{q-r}$ by \eqref{e110} and \eqref{e254} to obtain an upper bound $C^{q}D(u)^{q-r}$. Then we can complete the estimate by
\begin{equation*}
\sum\limits_{s_{0}-t_{m_{0}}=1}^{u}\sum\limits_{x_{0}\in\tilde{I}_{0}}p_{s_{0}-t_{m_{0}-1}}(S_{t_{m_{0}}},x_{0})\leq u
\end{equation*}
and
\begin{equation*}
\sum\limits_{0<t_{1}<\cdots<t_{m_{0}}}\mathbf{P}((t_{1},\cdots,t_{m_{0}}),(S_{t_{1}},\cdots,S_{t_{m_{0}}}))\leq(CD(u))^{m_{0}-1}.
\end{equation*}

{\bf Case B2}:

If $k>0$ in \eqref{e2a}, then we have $\prod_{i=k-1}^{1}\sum_{(x_{k-2},\cdots,x_{0})\in(\tilde{I}_{0})^{k-1}}(p_{s_{i}-s_{i-1}}(x_{i-1},x_{i}))^{2}$ and for any $x_{k-1}$, we can sum over $s_{0},\cdots,s_{k-2}$ and $x_{0},\cdots,x_{k-2}$ by \eqref{e110} and \eqref{e254} (hold $(s_{k-1},x_{k-1})$ for the moment), which gives $C^{q}D(u)^{k-1}$. For the same reason, then we can sum over $s_{k+1},\cdots,s_{q-r}$ and $x_{k+1},\cdots,x_{q-r}$ (hold $(s_{k},x_{k})$ for the moment), which gives $C^{q}D(u)^{q-r-k}$. These summations and products together give $C^{q}D(u)^{q-r-1}$, and then we can complete all the estimate by bounding
\begin{equation*}
\begin{split}
\sum\limits_{s_{k}-s_{k-1}=1}^{m_{k}u}&\sum\limits_{s_{k-1}<t_{1}<\cdots<t_{m_{k}}<s_{k}}\sum\limits_{x_{k-1},x_{k}\in\tilde{I}_{0}}\\
&p_{s_{k}-s_{k-1}}(x_{k-1},x_{k})\mathbf{P}((s_{k-1},t_{1},\cdots,t_{m_{k}},s_{k}),(x_{k-1},S_{t_{1}},\cdots,S_{t_{m_{k}}},x_{k})).
\end{split}
\end{equation*}
via \eqref{e255}-\eqref{e257}.

According to the upper bounds in {\bf Case A} and {\bf Case B}, we can obtain an upper bound $C^{q}q^{2}(qu)^{1+\epsilon^{4}}lD(u)^{q+r}$ for \eqref{e2a} by summing over $m_{k}$ and $m'_{k'}$. Recall that our analysis in {\bf Case A} and {\bf Case B} is based on $1\leq r\leq q-1$. Hence, for $1\leq r\leq q-1$, we can sum over $k$ and $k'$ to bound \eqref{e242} from above by $C^{q}u^{1+\epsilon^{4}}lD(u)^{q+r}$, since $q^{5+\epsilon^{2}}\ll C^{q}$.

It still remains to bound the case $r=q$ in \eqref{e242}, where $\underline{s}=\{s_{0}\}$. This is relatively simple. We use the expression in the first line of \eqref{e242}. Suppose that the $t$-index right beside $s_{0}$ is $t_{j}$. Without loss of generality, we may assume $s_{0}<t_{j}$. Then we have
\begin{equation*}
\sum\limits_{t_{j}-s_{0}=1}^{u}\sum\limits_{\tilde{x}_{0}\in\tilde{I}_{0}}p_{t_{j}-s_{0}}(x_{0},S_{t_{j}})\leq u.
\end{equation*}
For the other $t,t'$-indices, we just use the trivial bound
\begin{equation*}
\sum\limits_{t=1}^{u}p_{t}(0,S_{t})\leq D(u)
\end{equation*}
and then we obtain an upper bound $C^{q}ulD(u)^{q+r-1}$ for the case $r=q$ in \eqref{e242}.

Finally, we substitute everything into \eqref{e241} and by recalling $\lambda'(\beta)\sim\beta$, $\beta^{2}D(u)<(1+2\epsilon)$, we have
\begin{equation*}
\begin{split}
\mathbb{V}\mbox{ar}^{S}(X)&\leq(1+\epsilon^{3}/2)^{q+1}+\frac{C^{q}u^{1+\epsilon^{4}} }{2Ra_{l}}\sum\limits_{r=1}^{q}(1+2\epsilon)^{r}\\
&\leq(1+\epsilon^{3}/2)^{q+1}+\frac{q(2C)^{q}}{2R}l^{-\epsilon^{3}}\\
&\leq(1+\epsilon^{3}/2)^{q+1}+1\leq(1+\epsilon^{3})^{q}
\end{split}
\end{equation*}
and we conclude Lemma \ref{l205}.
\end{proof}
\section{Proof of Theorem \ref{t104}}\label{s300}
In this proof, for any given $\beta$ and $\epsilon$, we will estimate the partition function at a special time $N$, defined by
\begin{equation}\label{e301}
N_{\beta,\epsilon}:=\max\limits_{n}\{D(n)\leq(1-\epsilon)/\beta^{2}\}.
\end{equation}
By \cite[Proposition 2.5]{CSY03}, we have
\begin{equation*}
p(\beta)=\sup\limits_{N}\frac{1}{N}\mathbb{E}[\log\hat{Z}_{N,\beta}^{\omega}]\geq\frac{1}{N_{\beta,\epsilon}}\mathbb{E}[\log \hat{Z}_{N_{\beta,\epsilon},\beta}^{\omega}].
\end{equation*}
To simplify the notation, we will use $N$ as $N_{\beta,\epsilon}$ in the following without any ambiguity. We may emphasize several times that the choice of $N$ satisfies \eqref{e301}.

To show \eqref{e119}, we need to bound $\mathbb{E}[\log\hat{Z}_{N,\beta}^{\omega}]$ appropriately. The key ingredient of the proof is the following result proved in \cite{CTT15}. Here we cite a version stated in \cite{BL17}.
\begin{proposition}[{\hspace{1sp}\cite[Proposition 4.3]{BL17}}]\label{p301}
For any $m\in\mathbb{N}$ and any random vector $\eta = (\eta_{1},\cdots,\eta_{m})$ which satisfies the property that there exists a constant $K>0$ such that
\begin{equation}\label{e303}
\mathbb{P}(|\eta|\leq K)=1.
\end{equation}
Then for any convex function $f$, we can find a constant $C_{1}$ uniformly for $m$, $\eta$ and $f$, such that for any $a$, $M$ and any positive $t>0$, the inequality
\begin{equation}\label{e304}
\mathbb{P}\left(f(\eta)\geq a, |\triangledown f(\eta)|\leq M\right)\mathbb{P}\left(f(\eta)\leq a-t\right)\leq2\exp\left(-\frac{t^{2}}{C_{1}K^{2}M^{2}}\right)
\end{equation}
holds, where $|\triangledown f|:=\sqrt{\sum\limits_{i=1}^{m}\left(\frac{\partial f}{\partial x_{i}}\right)^{2}}$ is the norm of the gradient of $f$.
\end{proposition}
We will apply Proposition \ref{p301} to $\log\hat{Z}_{N,\beta}^{\omega}$ and the environment $\omega$. However, this proposition is only valid for bounded and finite-dimension random vector. Since $\log\hat{Z}_{N,\beta}^{\omega}$ is a function of countable-dimension random field and $\omega$ may not be bounded, we need to restrict the range of the random walk $S$ so that $\log\hat{Z}_{N,\beta}^{\omega}$ is determined by finite many $\omega_{i,x}$'s and respectively, truncate $\omega$ so that it is finite.

First, we define a subset of $\mathbb{N}\times\mathbb{Z}$ by
\begin{equation*}
\mathcal{T}=\mathcal{T}_{N}:=\{(n,x): 1\leq n\leq N, |x-b_{N}|\leq Ra_{N}\},
\end{equation*}
where $R$ is a constant that will be determined later and $a_{N}, b_{N}$ has been introduced in \eqref{e102}. We will choose $R$ large enough so that the trajectory of $S$ up to time $N$ entirely falls in $\mathcal{T}$ with probability close to $1$ for any $N=N_{\beta,\epsilon}$. We can also assume that $a_{N}$ is an integer without loss of generality.

Then we define
\begin{equation}\label{e306}
\bar{Z}_{N,\beta}^{\omega}:=\mathbf{E}\left[\exp\left(\beta\sum\limits_{n=1}^{N}\omega_{n,S_{n}}-N\lambda(\beta)\right)
\mathbbm{1}_{\{S\in\mathcal{T}\}}\right],
\end{equation}
where $\{S\in\mathcal{T}\}:=\{S:(n,S_{n})\in\mathcal{T},\quad\forall1\leq n\leq N\}$. Note that $\bar{Z}_{N,\beta}^{\omega}\leq\hat{Z}_{N\beta}^{\omega}$. Readers may check that $\log\bar{Z}_{N,\beta}^{\omega}$ is indeed a finite-dimension convex function and hence, we can apply Proposition \ref{p301} to $\log\bar{Z}_{N,\beta}$. Since our goal is to find a lower bound for $\mathbb{E}[\log\hat{Z}_{N,\beta}^{\omega}]$, we can first estimate the left tail of $\log\bar{Z}_{N,\beta}^{\omega}$, which can be done by bounding the first probability on the left-hand side of \eqref{e304} from below.

We show the following result.
\begin{lemma}\label{l302}
For arbitrarily small $\epsilon>0$, there exist $\beta_{\epsilon}$ and $M=M_{\epsilon}$, such that for any $\beta\in(0,\beta_{\epsilon})$, it follows that
\begin{equation}\label{e307}
\mathbb{P}\left(\bar{Z}_{N,\beta}^{\omega}\geq\frac{1}{2},\big|\triangledown\log\bar{Z}_{N_{\beta,\epsilon},\beta}^{\omega}\big|\leq M\right)\geq\frac{\epsilon}{100}.
\end{equation}
\end{lemma}
To prove Lemma \ref{l302}, we need a result from \cite{BL16}, which we state as
\begin{lemma}[{\hspace{1sp}\cite[Lemma 6.4]{BL16}}]\label{l303}
For any $\epsilon>0$, if $\beta$ is sufficiently small such that $N=N_{\beta,\epsilon}$ is large enough, then
\begin{equation}\label{e308}
\mathbb{E}[(\hat{Z}_{N,\beta}^{\omega})^{2}]\leq\frac{10}{\epsilon}
\end{equation}
\end{lemma}
\begin{proof}[Proof of Lemma \ref{l302}]
By Lemma \ref{l303} and the fact $\bar{Z}_{N,\beta}^{\omega}\leq\hat{Z}_{N,\beta}^{\omega}$,
\begin{equation*}
\mathbb{E}[(\bar{Z}_{N,\beta}^{\omega})^{2}]\leq\frac{10}{\epsilon}.
\end{equation*}
Then by Paley-Zygmund inequality, we have
\begin{equation*}
\mathbb{P}\left(\bar{Z}_{N,\beta}^{\omega}\geq\frac{1}{2}\right)\geq\frac{\left(\mathbf{P}(S\in\mathcal{T})-\frac{1}{2}\right)^{2}}{\mathbb{E}[(\bar{Z}_{N,\beta}^{\omega})^{2}]}\geq\frac{\epsilon}{50},
\end{equation*}
where the last inequality holds by choosing $R$ large enough in $\mathcal{T}$.

By using notation
\begin{equation*}
f(\omega):=\log\bar{Z}_{N,\beta}^{\omega},
\end{equation*}
we have
\begin{equation}\label{e312}
\begin{split}
\mathbb{P}\left(\bar{Z}_{N,\beta}^{\omega}\geq\frac{1}{2},|\triangledown f(\omega)|\leq M\right)&=\mathbb{P}\left(\bar{Z}_{N,\beta}^{\omega}\geq\frac{1}{2}\right)-\mathbb{P}\left(\bar{Z}_{N,\beta}^{\omega}\geq\frac{1}{2},|\triangledown f(\omega)|>M\right)\\
&\geq\frac{\epsilon}{50}-\frac{1}{M^{2}}\mathbb{E}\left[|\triangledown f(\omega)|^{2}\mathbbm{1}_{\{\bar{Z}_{N,\beta}^{\omega}\geq\frac{1}{2}\}}\right].
\end{split}
\end{equation}
To compute $\triangledown f(\omega)$, we find that
\begin{equation*}
\begin{split}
\frac{\partial}{\partial\omega_{k,x}}\log\bar{Z}_{N,\beta}^{\omega}&=\frac{\beta}{\bar{Z}_{N,\beta}^{\omega}}\mathbf{E}\left[\exp\left(\beta\sum\limits_{n=1}^{N}\omega_{n,S_{n}}-N\lambda(\beta)\right)\mathbbm{1}_{\{S_{k}=x,S\in\mathcal{T}\}}\right]\\
&\leq\frac{\beta}{\bar{Z}_{N,\beta}}\mathbf{E}\left[\exp\left(\beta\sum\limits_{n=1}^{N}\omega_{n,S_{n}}-N\lambda(\beta)\right)\mathbbm{1}_{\{S_{k}=x\}}\right].
\end{split}
\end{equation*}
Then
\begin{equation*}
\begin{split}
|\triangledown f(\omega)|^{2}&=\sum\limits_{(k,x)\in\mathcal{T}}\left|\frac{\partial}{\partial\omega_{k,x}}\log\bar{Z}_{N,\beta}^{\omega}\right|^{2}\\
&\leq\frac{\beta^{2}}{(\bar{Z}_{N,\beta}^{\omega})^{2}}\sum\limits_{k=1}^{N}\sum\limits_{x\in\mathbb{Z}}\left(\mathbf{E}\left[\exp\left(\beta\sum\limits_{n=1}^{N}\omega_{n,S_{n}}-N\lambda(\beta)\right)\mathbbm{1}_{\{S_{k}=x\}}\right]\right)^{2}.
\end{split}
\end{equation*}
Note that
\begin{equation*}
\begin{split}
&\left(\mathbf{E}\left[\exp\left(\beta\sum\limits_{n=1}^{N}\omega_{n,S_{n}}-N\lambda(\beta)\right)\mathbbm{1}_{\{S_{k}=x\}}\right]\right)^{2}\\
=&\mathbf{E}^{\bigotimes2}\left[\exp\left(\beta\sum\limits_{n=1}^{N}(\omega_{n,S_{n}}+\omega_{n,\tilde{S}_{n}})-2N\lambda(\beta)\right)\mathbbm{1}_{\{S_{k}=\tilde{S}_{k}=x\}}\right].
\end{split}
\end{equation*}
Therefore,
\begin{equation*}
|\triangledown f(\omega)|^{2}\leq\frac{\beta^{2}}{(\bar{Z}_{N,\beta}^{\omega})^{2}}\mathbf{E}^{\bigotimes2}\left[\sum\limits_{k=1}^{N}\mathbbm{1}_{\{S_{k}=\tilde{S}_{k}\}}\exp\left(\beta\sum\limits_{n=1}^{N}(\omega_{n,S_{n}}+\omega_{n,\tilde{S}_{n}})
-2N\lambda(\beta)\right)\right]
\end{equation*}
Then we have
\begin{equation}\label{e317}
\mathbb{E}\left[|\triangledown f(\omega)|^{2}\mathbbm{1}_{\{\bar{Z}_{N,\beta}^{\omega}\geq\frac{1}{2}\}}\right]\leq4\mathbf{E}^{\bigotimes2}\left[\beta^{2}\sum\limits_{k=1}^{N}\mathbbm{1}_{\{S_{k}=\tilde{S}_{k}\}}\exp\left(\gamma(\beta)\sum\limits_{n=1}^{N}\mathbbm{1}_{\{S_{n}=\tilde{S}_{n}\}}\right)\right],
\end{equation}
where
\begin{equation*}
\gamma(\beta):=\lambda(2\beta)-2\lambda(\beta).
\end{equation*}
We denote
\begin{equation*}
Y:=\sum\limits_{n=1}^{N}\mathbbm{1}_{\{S_{n}=\tilde{S}_{n}\}}
\end{equation*}
for short. It is not hard to check that
\begin{equation*}
\lambda(2\beta)-2\lambda(\beta)\sim\beta^{2},\quad\mbox{as}~\beta\to0.
\end{equation*}
Hence, when $\beta$ is sufficiently small, we have
\begin{equation}\label{e321}
\begin{split}
&\mathbf{E}^{\bigotimes2}\left[\beta^{2}\sum\limits_{k=1}^{N}\mathbbm{1}_{\{S_{k}=\tilde{S}_{k}\}}\exp\left(\gamma(\beta)\sum\limits_{n=1}^{N}\mathbbm{1}_{\{S_{n}=\tilde{S}_{n}\}}\right)\right]\\
\leq&\mathbf{E}^{\bigotimes2}\left[\beta^{2}Y\exp((1+\epsilon^{3})\beta^{2}Y)\right]\leq\mathbf{E}^{\bigotimes2}\left[C_{\epsilon}\exp((1+\epsilon^{2})\beta^{2}Y)\right],
\end{split}
\end{equation}
where $C_{\epsilon}$ is a constant such that
\begin{equation*}
x\exp((1+\epsilon^{3})x)\leq C_{\epsilon}\exp((1+\epsilon^{2})x),\quad\forall x\geq0.
\end{equation*}
Again by Lemma \ref{l303},
\begin{equation}\label{e323}
\mathbf{E}^{\bigotimes2}\left[\beta^{2}\sum\limits_{k=1}^{N}\mathbbm{1}_{\{S_{k}=\tilde{S}_{k}\}}\exp\left(\gamma(\beta)\sum\limits_{n=1}^{N}\mathbbm{1}_{\{S_{n}=\tilde{S}_{n}\}}\right)\right]\leq\frac{10C_{\epsilon}}{\epsilon}.
\end{equation}
We can choose $M=M_{\epsilon}=20\sqrt{10C_{\epsilon}}/\epsilon^{2}$ and then combine \eqref{e312}, \eqref{e317}, \eqref{e321}, \eqref{e323}, we then conclude Lemma \ref{l302}.
\end{proof}
Finally, we can now prove Theorem \ref{t104}. Readers should keep in mind that $N=N_{\beta,\epsilon}$.
\begin{proof}[Proof of Theorem \ref{t104}]
Because the environment $\omega$ has a finite moment generating function, we can find some positive constants $C_{2}$ and $C_{3}$, such that
\begin{equation*}
\mathbb{P}(|\omega_{1,0}|\geq t)\leq C_{2}\exp(-C_{3}t).
\end{equation*}
Note that we will focus on the environment with index in $\mathcal{T}$. We can estimate that
\begin{equation}\label{e325}
\mathbb{P}\left(\max\limits_{(n,x)\in\mathcal{T}}|\omega_{n,x}|\geq t\right)\leq C_{4}Na_{N}\exp(-C_{3}t).
\end{equation}
Note that
\begin{equation*}
\left\{\max\limits_{(n,x)\in\mathcal{T}}|\omega_{n,x}|<t\right\}\subset\left\{\omega_{n,x}>-t,\quad\forall (n,x)\in\mathcal{T}\right\}
\end{equation*}
and recall the definition of $\bar{Z}_{N,\beta}^{\omega}$ from \eqref{e306}, then we obtain a rough bound
\begin{equation}\label{e327}
\mathbb{P}\left(\log\bar{Z}_{N,\beta}^{\omega}<-(\beta t+\lambda(\beta))N\right)\leq C_{4}Na_{N}\exp(-C_{3}t).
\end{equation}
We will use \eqref{e327} later to bound the left tail of $\log\hat{Z}_{N,\beta}^{\omega}$ for large $t$.

In order to apply Proposition \ref{p301}, we need to truncate the environment appropriately. We set $\tilde{\omega}_{n,x}:=\omega_{n,x}\mathbbm{1}_{\{|\omega_{n,x}|\leq(\log N)^{2}\}}$ and define
\begin{equation*}
f(\tilde{\omega}):=\log\mathbf{E}\left[\exp\left(\beta\sum\limits_{n=1}^{N}\tilde{\omega}_{n,S_{n}}-N\lambda(\beta)\right)\mathbbm{1}	_{\{S\in\mathcal{T}\}}\right].
\end{equation*}
Then
\begin{equation*}
\begin{split}
&\mathbb{P}\left(\bar{Z}_{N,\beta}^{\omega}\geq\frac{1}{2},|\triangledown\log\bar{Z}_{N,\beta}^{\omega}|\leq M\right)\\
=&\mathbb{P}\left(\bar{Z}_{N,\beta}^{\omega}\geq\frac{1}{2},|\triangledown\log\bar{Z}_{N,\beta}^{\omega}|\leq M,\tilde{\omega}=\omega\right)+\mathbb{P}\left(\bar{Z}_{N,\beta}^{\omega}\geq\frac{1}{2},|\triangledown\log\bar{Z}_{N,\beta}^{\omega}|\leq M,\tilde{\omega}\neq\omega\right)\\
\leq&\mathbb{P}\left(f(\tilde{\omega})\geq-\log2,|\triangledown f(\tilde{\omega})|\leq M\right)+\mathbb{P}(\tilde{\omega}\neq\omega)
\end{split}
\end{equation*}
By Lemma \ref{l302} and \eqref{e325},
\begin{equation*}
\begin{split}
&\mathbb{P}\left(f(\tilde{\omega})\geq-\log2,|\triangledown f(\tilde{\omega})|\leq M\right)\\
\geq&\mathbb{P}\left(\bar{Z}_{N,\beta}^{\omega}\geq\frac{1}{2},|\triangledown\log\bar{Z}_{N,\beta}^{\omega}|\leq M\right)-\mathbb{P}(\tilde{\omega}\neq\omega)\\
\geq&\frac{\epsilon}{100}-C_{4}Na_{N}\exp(-C_{3}(\log N)^{2})\geq\frac{\epsilon}{200},
\end{split}
\end{equation*}
where the last inequality holds for large $N$, i.e., for small $\beta$. Now we apply Proposition \ref{p301} to $f(\tilde{\omega})$ and we obtain
\begin{equation*}
\mathbb{P}\left(f(\tilde{\omega})\leq-\log2-t\right)\leq\frac{400}{\epsilon}\exp\left(-\frac{t^{2}}{C_{1}(\log N)^{4}M^{2}}\right).
\end{equation*}
Finally,
\begin{equation}\label{e332}
\begin{split}
&\mathbb{P}\left(\log\bar{Z}_{N,\beta}^{\omega}\leq-\log2-t\right)\\
=&\mathbb{P}\left(\log\bar{Z}_{N,\beta}^{\omega}\leq-\log2-t,\tilde{\omega}=\omega\right)+\mathbb{P}\left(\log\bar{Z}_{N,\beta}^{\omega}\leq-\log2-t,\tilde{\omega}\neq\omega\right)\\
\leq&\mathbb{P}\left(f(\tilde{\omega})\leq-\log2-t\right)+\mathbb{P}(\tilde{\omega}\neq\omega)\\
\leq&\frac{400}{\epsilon}\exp\left(-\frac{t^{2}}{C_{1}(\log N)^{4}M^{2}}\right)+C_{4}Na_{N}\exp(-C_{3}(\log N)^{2}).
\end{split}
\end{equation}
We can now bounded the left tail of $\log\hat{Z}_{N,\beta}^{\omega}$. Since it is larger than $\log\bar{Z}_{N,\beta}^{\omega}$, we can rewrite \eqref{e327} and \eqref{e332} as
\begin{equation}\label{e333}
\mathbb{P}\left(\log\hat{Z}_{N,\beta}^{\omega}<-(\beta t+\lambda(\beta))N\right)\leq C_{4}Na_{N}\exp(-C_{3}t)
\end{equation}
and respectively,
\begin{equation}\label{e334}
\begin{split}
&\mathbb{P}\left(\log\hat{Z}_{N,\beta}^{\omega}\leq-\log2-t\right)\\
\leq&\frac{400}{\epsilon}\exp\left(-\frac{t^{2}}{C_{1}(\log N)^{4}M^{2}}\right)+C_{4}Na_{N}\exp(-C_{3}(\log N)^{2}).
\end{split}
\end{equation}
For $\log\hat{Z}_{N,\beta}^{\omega}$ with large negative value (for example, it is less than $-N^{2}$), we use the bound \eqref{e333}, which shows that the mass of $\log\hat{Z}_{N,\beta}^{\omega}$ on $(-N^{2},-\infty)$ can be bounded below by some constant $-C$. For $\log Z_{N, \beta}^{\omega}$ with small negative value, we use the bound \eqref{e334}, which shows that the mass of $\log\hat{Z}_{N,\beta}^{\omega}$ is bounded below by $-\tilde{C}_{\epsilon}(\log N)^{2}$ with some constant $\tilde{C}_{\epsilon}$. Therefore, we obtain
\begin{equation*}
p(\beta)\geq\frac{1}{N}\mathbb{E}\left[\log\hat{Z}_{N,\beta}^{\omega}\right]\geq-\frac{C_{5,\epsilon}(\log N)^{2}}{N}\geq-\frac{1}{D^{-1}\left((1-\epsilon)/\beta^{2}\right)^{1-\epsilon}}
\end{equation*}
for $\beta$ small enough, where the last inequality is due to the definition of $N=N_{\beta,\epsilon}$.
\end{proof}
\textbf{Acknowledgements:}~The author would like to acknowledge support from AcRF Tier 1 grant R-146-000-220-112. The author also wants to thank Professor Rongfeng Sun for introducing and discussing this topic, and helping revise this paper. The author is especially grateful to Professor Quentin Berger for sharing his manuscript \cite{Ber17}, in particular, the proof of Theorem \ref{t102} before publication. The author also thanks Professor Francesco Caravenna, who helped the author obtain more insight into this problem when he visited Singapore. Finally, the author would like to thank the unknown referees, who helped the author remove some unnatural assumptions on the underlying random walk $S$ and improve the quality of this paper.
\bibliographystyle{plain}
\bibliography{references}
\end{document}